\documentclass[a4paper,reqno,11pt]{amsart}
\usepackage[utf8]{inputenc}
\usepackage{amsfonts}
\usepackage{amsmath}
\usepackage{amssymb}
\usepackage{float}
\usepackage{tikz-cd}
\usepackage{graphicx}
\usepackage{bbm}
\usepackage{amscd}
\usepackage[sort&compress,numbers]{natbib}
\usepackage[left=2cm,right=2cm,bottom=3cm,top=3cm]{geometry}
\usepackage{etoolbox}
\usepackage{filecontents}
\usepackage{enumerate}
\usepackage[hidelinks]{hyperref}
\usepackage{tikz}

\usetikzlibrary{matrix}
\newcommand{\ga}{\gamma}
\newcommand{\acts}{\curvearrowright}

\newcommand{\eit}{\end{itemize}}

\newcommand{\mc}[1]{\mathcal{#1}}
\newcommand{\mbb}[1]{\mathbb{#1}}
\newcommand{\mr}[1]{\mathrm{#1}}

\newcommand{\mbf}[1]{\mathbf{#1}}

\makeatletter
\patchcmd{\BR@backref}{\newblock}{\newblock(page~}{}{}
\patchcmd{\BR@backref}{\par}{)\par}{}{}
\makeatother
\theoremstyle{plain}
\newtheorem{theorem}{Theorem}[section]
\newtheorem*{theorem*}{Theorem}
\newtheorem{proposition}[theorem]{Proposition}

\newtheorem{corollary}[theorem]{Corollary}

\newtheorem{example}[theorem]{Example}

\newtheorem{lemma}[theorem]{Lemma}

\theoremstyle{definition}
\newtheorem{definition}[theorem]{Definition}

\theoremstyle{remark}

\newtheorem{remark}[theorem]{Remark}
\newtheorem{claim}{Claim}[theorem]
\newtheorem{pclaim}{Claim}[theorem]

\newcommand{\defi}{\begin{definition}}
\newcommand{\fdefi}{\end{definition}}
\newcommand{\eje}{\begin{example}}
\newcommand{\feje}{\end{example}}
\newcommand{\ejes}{\begin{ejemplos}}
\newcommand{\fejes}{\end{ejemplos}}
\newcommand{\lema}{\begin{lemma}}
\newcommand{\flema}{\end{lemma}}
\newcommand{\teor}{\begin{theorem}}
\newcommand{\fteor}{\end{theorem}}
\newcommand{\nota}{\begin{remark}}
\newcommand{\fnota}{ \end{remark}}
\newcommand{\clam}{\begin{claim}}
\newcommand{\fclam}{\end{claim}}
\newcommand{\pclam}{\begin{pclaim}}
\newcommand{\fpclam}{\end{pclaim}}
\newcommand{\clams}{\begin{claim*}}
\newcommand{\fclams}{\end{claim*}}
\newcommand{\prop}{\begin{proposition}}
\newcommand{\fprop}{\end{proposition}}
\newcommand{\cor}{\begin{corollary}}
\newcommand{\fcor}{\end{corollary}}

\newcommand{\lclam}{\begin{lclaim}}
\newcommand{\flclam}{\end{lclaim}}
\newcommand{\prucl}{\prue[Proof of Claim:]}
\newcommand{\fprucl}{\fprue}
\newcommand{\ben}{\begin{enumerate}}
\newcommand{\een}{\end{enumerate}}
\newcommand{\bit}{\begin{itemize}}

\newcommand{\N}{\mathbb N}
\newcommand{\R}{\mathbb R}

\DeclareMathOperator{\Quo}{\mathrm{Quo}}

\DeclareMathOperator{\flim}{\mathrm{FLim}}
\DeclareMathOperator{\iso}{\mathrm{Iso}}
\DeclareMathOperator{\age}{\mathrm{Age}}
\DeclareMathOperator{\Emb}{\mathrm{Emb}}
\DeclareMathOperator{\Eemb}{\mathcal{E}\mathrm{mb}}
\DeclareMathOperator{\Aut}{\mathrm{Aut}}

\DeclareMathOperator{\osc}{\mathrm{Osc}}
\DeclareMathOperator{\emb}{\mathrm{emb}}
\DeclareMathOperator{\ball}{\mathrm{Ball}}
\DeclareMathOperator{\aut}{\mathrm{Aut}}
\newcommand{\rest}{\upharpoonright}

\newcommand{\C}{{\mathbb C}}
\newcommand{\sig}{\sigma}

\newcommand{\vphi}{\varphi}

\newcommand{\id}{\mr{Id}}
\newcommand{\conj}[2]{ \{ {#1}\,:\,{#2} \} }
\newcommand{\prue}{\begin{proof}}
\newcommand{\fprue}{\end{proof}}

\newcommand{\al}{\alpha}

\newcommand{\Om}{\Omega}
\newcommand{\de}{\delta}
\newcommand{\De}{\Delta}
\newcommand{\la}{\lambda}
\newcommand{\ro}{\varrho}
\DeclareMathOperator{\im}{\text{Im}}
\DeclareMathOperator{\Tr}{\text{Tr}}
\newcommand{\posy}{\pmb{A({ \mbb P})}}
\newcommand{\nrm}[1]{\|#1\|}
\newcommand{\con}{\subseteq}
\newcommand{\vep}{\varepsilon}
\newcommand{\buit}{\emptyset}
\newcommand{\fin}{\textsc{FIN}}

\newcommand{\pe}{\prec}

\newcommand{\quo}{/\hspace{-0.1cm}/}
\makeatletter
\newtheorem*{rep@theorem}{\rep@title}
\newcommand{\newreptheorem}[2]{%
\newenvironment{rep#1}[1]{%
 \def\rep@title{\bf #2 \ref{##1}}%
 \begin{rep@theorem}}%
 {\end{rep@theorem}}}
\makeatother
\newreptheorem{theorem}{Theorem}

\begin{document}
\def\cprime{$'$}

\title[Ramsey property for Banach spaces and Choquet simplices]{The   Ramsey property for Banach spaces and Choquet
simplices}
\author[D. Barto\v{s}ov\'{a}]{Dana Barto\v{s}ov\'{a}}
\address{Department of Mathematical Sciences, Carnegie Mellon University,
Pennsylvania, USA}
\email{dbartoso@andrew.cmu.edu}
\author[J. Lopez-Abad]{Jordi Lopez-Abad}
\address{Departamento de Matem\'{a}ticas Fundamentales,
Facultad de Ciencias, UNED, 28040 Madrid, Spain}
\email{abad@mat.uned.es}
\author[M. Lupini]{Martino Lupini}
\address{School of Mathematics and Statistics\\
Victoria University of Wellington PO Box 600 \newline 
Wellington 6140
New Zealand}
\address{Mathematics Department\\
California Institute of Technology 
1200 E. California Blvd\newline
MC 253-37\\
Pasadena, CA 91125}
\email{martino.lupini@vuw.ac.nz}
\urladdr{http://www.lupini.org/}
\author[B. Mbombo]{Brice Mbombo}
\address{Department of Mathematics and Statistics, University of Ottawa,
Ottawa, ON, K1N 6N5, Canada}
\email{bmbombod@uottawa.ca}
\thanks{D.B. was supported by the grant FAPESP 2013/14458-9. J.L.-A. \ was partially supported by the grant MTM2012-31286 (Spain), the Fapesp Grant 2013/24827-1 (Brazil) and projet ANR AGRUME ANR-17-CE40-0026 (France). M.L.\ was partially supported by
the NSF Grant DMS-1600186. B. Mbombo was supported by Funda\c{c}\~{a}o de
Amparo \`{a} Pesquisa do Estado de S\~{a}o Paulo (FAPESP) postdoctoral
grant, processo 12/20084-1. This work was initiated during a visit of
J.L.-A.\ to the Universidade de Sao P\~{a}ulo in 2014, and continued during
visits of D.B.\ and J.L.-A. to the Fields Institute in the Fall 2014, a visit
of M.L.\ to the Instituto de Ciencias Matem\'{a}ticas in the Spring 2015, and
a visit of all the authors at the Banff International Research Station in
occasion of the Workshop on Homogeneous Structures in the Fall 2015. The
hospitality of all these institutions is gratefully acknowledged.}

\begin{abstract}
We show that the Gurarij space $\mathbb{G}$   has extremely amenable automorphism group. This answers a question of Melleray and Tsankov. We also
compute the universal minimal flow of the automorphism group of the
Poulsen simplex $\mathbb{P}$ and we prove that it consists of  the canonical action on $\mathbb{P}$ itself.  This answers a question of
Conley and T\"{o}rnquist. We show that the pointwise stabilizer of any
closed proper face of $\mathbb{P}$ is extremely amenable. Similarly, the
pointwise stabilizer of any closed proper biface of the unit ball of the
dual of the Gurarij space (the Lusky simplex) is extremely amenable.

These results are obtained via several Kechris--Pestov--Todorcevic
correspondences, by establishing the approximate Ramsey property for several
classes of finite-dimensional Banach spaces and function  systems and their versions with distinguished contractions.  This is the first direct application of
the Kechris--Pestov--Todorcevic correspondence in the setting of metric
structures. The fundamental combinatorial principle that underpins the
proofs is the Dual Ramsey Theorem of Graham and Rothschild.
%

\end{abstract}

\maketitle





%

\section{Introduction}
Given a topological group $G$, a compact $G$-space or $G$-flow is a compact
Hausdorff space $X$ endowed with a continuous action of $G$. Such a $G$-flow 
$X$ is called minimal when every orbit is dense. There is a natural notion
of morphism between $G$-flows, given by a   $G$%
-equivariant continuous map (factor). A minimal $G$-flow is universal if it
factors onto any   minimal $G$-flow. It is a classical fact that any
topological group $G$ admits a unique (up to isomorphism of $G$-flows)
universal minimal flow, usually denoted by $M(G)$ \cite%
{ellis_universal_1960,gutman_new_2013}. For any locally compact non compact
Polish group $G$, the universal minimal $G$-flow is nonmetrizable. At the
opposite end, non locally compact topological groups often have metrizable
universal minimal flows, or even reduced to a single point. A topological
group for which $M(G)$ is a singleton is called extremely amenable.
(Amenability of $G$ is equivalent to the assertion that every compact $G$%
-space has an invariant Borel measure. Thus any extremely amenable group is,
in particular, amenable.)

The universal minimal flow has been explicitly computed for  a number of 
topological groups, typically given as automorphism groups of naturally
arising mathematical structures. Examples of extremely amenable Polish
groups include the group of order automorphisms of $\mathbb{Q}$ \cite%
{pestov_free_1998}, the group of unitary operators on the separable
infinite-dimensional Hilbert space \cite{gromov_topological_1983}, the
automorphism group of the hyperfinite II$_{1}$ factor and of infinite type
UHF $C^{\ast }$-algebras \cite{giordano_extremely_2007,eagle_fraisse_2014},
or the isometry group of the Urysohn space \cite{pestov_ramsey-milman_2002}.
Examples of nontrivial metrizable universal minimal flows include the
universal minimal flow of the group of orientation preserving homeomorphisms
of the circle, which is equivariantly homeomorphic to the circle itself \cite%
{pestov_free_1998}, the universal minimal flow of the group $S_{\infty }$ of
permutations of $\mathbb{N}$, which can be identified with the space of
linear orders on $\mathbb{N}$ \cite{glasner_minimal_2002}, and the universal
minimal flow of the homeomorphism group $\mathrm{Homeo}(2^{\mathbb{N}})$ of
the Cantor set $2^{\mathbb{N}}$, which can be seen as the canonical action
of $\mathrm{Homeo}(2^{\mathbb{N}})$ on the space of maximal chains of closed
subsets of $2^{\mathbb{N}}$ \cite%
{uspenskij_universal_2000,glasner_universal_2003,kechris_fraisse_2005}.

There are essentially two known ways to establish extreme amenability of a
given topological group. The first method involves the phenomenon of \emph{%
concentration of measure}, and can be applied to topological groups that
admit an increasing sequence of compact subgroups with a  dense union \cite[%
Chapter 4]{pestov_dynamics_2006,gromov_topological_1983}. The second method applies to automorphism
groups of discrete ultrahomogeneous structures or, more generally,
approximately ultrahomogeneous metric structures \cite[Chapter 6]%
{pestov_dynamics_2006}. A metric structure is approximately ultrahomogeneous
if any partial isomorphism between finitely generated substructures is the
pointwise limit of maps that are restrictions of automorphisms. It is worth
noting that any Polish group can be realized as the automorphism group of an
approximately ultrahomogeneous metric structure \cite[Theorem 6]%
{melleray_note_2010}. For the automorphism group $\mathrm{Aut}(M)$ of an
approximately ultrahomogeneous structure $M$, extreme amenability is
equivalent to the approximate Ramsey property of the class of finitely
generated substructures of $M$. This criterion is known as the
Kechris--Pestov--Todorcevic (KPT) correspondence, first established in \cite%
{kechris_fraisse_2005} for discrete structures, and recently generalized to
the metric setting in \cite{melleray_extremely_2014}. The discrete KPT
correspondence has been used extensively in the last decade. In this paper
the KPT correspondence is directly used for the first time to obtain new
natural extreme amenability results.


In all the known examples of computations of metrizable universal minimal
flows, the argument hinges on extreme amenability of a suitable subgroup and
the following result due to Nguyen Van Th\'{e} \cite%
{nguyen_van_the_more_2013} based on previous work of Pestov \cite%
{pestov_free_1998}. Suppose that $G$ is a topological group with an
extremely amenable closed subgroup $H$. If the completion $X$ of the
homogeneous space $G/H$ endowed with the quotient of the right uniformity on 
$G$ is a minimal compact $G$-space, then $X$ is the universal minimal flow
of $G$. It was recently shown in \cite%
{melleray_polish_2016,ben_yaacov_metrizable_2017} that, whenever the
universal minimal flow of $G$ is metrizable, it can be realized as the
completion of $G/H$ for a suitable closed subgroup $H$ of $G$.

In this paper we compute the universal minimal flows of the automorphism
groups of structures coming from functional analysis and Choquet theory: the
Gurarij space $\mathbb{G}$ and the Poulsen simplex $\mathbb{P}$.
 Recall that the Gurarij space is the unique
separable {\em approximately ultrahomogeneous} Banach space that contains $\ell
_{n}^{\infty }$ for every $n\in \mathbb{N}$ \cite{lusky_gurarij_1976}, while 
$\mathbb{P}$ is the unique nontrivial metrizable Choquet simplex with dense
extreme boundary \cite{lindenstrauss_poulsen_1978}. The group $\mathrm{Aut}(%
\mathbb{G})$ of surjective linear isometries of the Gurarij space is shown
to be extremely amenable by establishing the approximate Ramsey property of
the class of finite-dimensional Banach spaces. This answers a question of
Melleray and Tsankov from \cite{melleray_extremely_2014} Similarly, the
stabilizer $\mathrm{Aut}_{p}(\mathbb{P})$ of an extreme point $p$ of $%
\mathbb{P}$ is proven to be extremely amenable by establishing the
approximate Ramsey property of the class of Choquet simplices with a
distinguished point. It is then deduced from this that the universal minimal
flow of $\mathrm{Aut}(\mathbb{P})$ is $\mathbb{P}$ itself, endowed with the
canonical action of $\mathrm{Aut}(\mathbb{P})$. This answers Question 4.4
from \cite{conley_fraisse_2017}. More generally, we prove that for any
closed face $F$ of $\mathbb{P}$, the pointwise stabilizer $\mathrm{Aut}_{F}(%
\mathbb{P})$ is extremely amenable. The analogous result holds in the Banach
space setting as well. A \emph{Lazar simplex} is a compact absolutely convex
set that arises as the unit ball of the dual of a Lindenstrauss space. The 
\emph{Lusky simplex }$\mathbb{L}$ is the Lazar simplex that arises in this
fashion from the Gurarij space. The group $\mathrm{Aut}(\mathbb{G})$ can be
identified with the group $\mathrm{Aut}(\mathbb{L})$ of symmetric affine
homeomorphisms of $\mathbb{L}$. It is proven in \cite[Theorem 1.2]%
{lupini_fraisse_2016} that $\mathbb{L}$ plays the same role among Lazar
simplices as the Poulsen simplex plays in the class of Choquet simplices,
where closed faces are replaced with closed bifaces. We prove that, for any
closed proper biface $H$ of $\mathbb{L}$, the corresponding pointwise
stabilizer $\mathrm{Aut}_{H}(\mathbb{L})$ is extremely amenable. In the
particular case when $H$ is the trivial biface, this recovers the extreme
amenability of $\mathrm{Aut}(\mathbb{G})$.

Recall that a \emph{function system} is a closed subspace $V$ of  the space of continuous $\mbb C$-valued functions $C(T)$ of some compact Hausdorff space $T$ containing the function constantly equal to 
$1$ and such that if $f\in V$ then the function $f^{\ast }$ defined by $%
f^{\ast }\left( t\right) =\overline{f\left( t\right) }$ also belongs to $V$.
In particular, when $K$ is a compact convex set, the  space $A(K)$ of continuous complex-valued affine functions on $K$   is a function system, and in fact any function system $%
V\subseteq C(T)$ arises in this way from a suitable compact convex set $K$.
Precisely, $K$ is the compact convex set of \emph{states }of $V$, that is,
the contractive functionals on $V$ that are \emph{unital}, i.e.\ map the
unit of $C(T)$ to $1$ (\cite[Theorem II.1.8]{alfsen_compact_1971}).
%
%
%
%
Furthermore, the map $K\mapsto A(K)$ is a contravariant isomorphism of
categories from the category of compact convex sets and continuous affine
maps, to the category of function systems and unital linear contractions (%
\emph{Kadison correspondence}). A metrizable compact convex set $K$ is a
simplex if and only if $A(K)$ is a separable Lindenstrauss space, which
means that the identity map of $A(K)$ is the pointwise limit of a sequence
of unital completely contractive maps that factor through finite-dimensional
(abelian) $C^{\ast }$-algebras. 
The function system $A(\mathbb{P})$ corresponding to the Poulsen simplex is the unique separable approximately
ultrahomogeneous function system that contains unital copies of $\ell
_{n}^{\infty }$ for $n\in \mathbb{N}$ \cite[Theorem 1.1]{lupini_fraisse_2016}%
. The automorphism group $\mathrm{Aut}(A(\mathbb{P}))$ can be identified
with the group of affine homeomorphisms of  $\mathbb{P}$. The Poulsen simplex 
$\mathbb{P}$ is then equivariantly homeomorphic to the state space of $A(%
\mathbb{P})$.

The main tool to establish the results mentioned above will be the Dual
Ramsey Theorem of Graham and Rothschild \cite{graham_ramseys_1971}. This is
a powerful pigeonhole principle known to imply many other results, such as
the Hales--Jewett theorem, and the Ramsey theorem. It can be seen to be
equivalent to a factorization result for colorings of Boolean matrices,
which implies the celebrated Graham--Leeb--Rothschild theorem on
Grassmannians over a finite field \cite{graham_ramseys_1972}.     In fact, it is shown in \cite{bartosova_ramsey_2017,bartosova_ramsey_2019} that  this
is again a particular case of a factorization result for colorings of
matrices over a finite field, stating that the coloring of matrices only
depends on the invertible matrix needed to transform a given matrix into one
in reduced column echelon form. In this paper, we provide factorization
theorems for colorings of matrices and Grassmannians over the real or
complex numbers. We prove in particular that colorings of matrices depend
only on the canonical norm that a given matrix determines, while colorings
of Grassmannians are determined by the Banach-Mazur type of the given
subspace.

The paper is organized as follows.  We start in Subsection \ref{njhjiuyrtt}   by recalling some basic concepts such as extreme
amenability. In Subsection \ref{oi9998555} we recall  and introduce different versions of ultrahomogeneity and Ramsey properties for Banach spaces, and we prove  a version of the KPT correspondence in this setting (Theorem \ref{liwjr3iwejirwe}).   In Subsection \ref{oi23349283} we  prove the approximate Ramsey property (ARP)  of the class $\{\ell_\infty^n\}_n$.   This has as a consequence the extreme amenability of the group of isometries of the Gurarij space. In Subsection \ref{uiuhyuuy893232321} we prove the (ARP) of the class of   polyhedral finite-dimensional spaces, and the class of all finite-dimensional Banach spaces.  Using this,  in Subsection \ref{uiuhyuuy893232326} we give a direct proof of the (ARP) for the class of finite metric spaces. This provides a new proof of extreme amenability of the isometry group of the Urysohn space  \cite{pestov_ramsey-milman_2002}.  Subsection \ref{lio3j4oirjir4488}  studies closed bifaces of Lusky simplices. We prove that  group stabilizers of closed proper bifaces of the Lazar simplex are extremely amenable. This is done by establishing the corresponding (ARP) and a (KPT)-correspondence, introduced in Subsubsection \ref{775939hhttti}.  In Section \ref{Sec:systems} we study Choquet simplices (with a distinguished face), and we prove that pointwise stabilizer of any closed proper face of the Poulsen simplex is extremely amenable. We conclude in Subsection  \ref{oi5t8t895466655} we prove that the universal minimal flow of the group of affine homeomorphisms of the Poulsen simplex $\mbb P$ is the canonical action on $\mbb P$.
%
%
%

\subsubsection*{Acknowledgments}

We are grateful to Ita\"{\i} Ben Yaacov, Clinton Conley, Valentin Ferenczi,
Alexander Kechris, Matt Kennedy, Julien Melleray, Lionel Nguyen Van Th\'{e},
Vladimir Pestov, Slawomir Solecki, Stevo Todorcevic, and Todor Tsankov for
several helpful conversations and remarks.


\section{The Ramsey property of Banach spaces}
\label{basic_notions}
The goal of this section is to introduce different notions of ``Ramsey property'' for several classes of structures. We show that, in the setting we are interested in, such notions are equivalent to each other. We furthermore establish an analogue of the Kechris--Pestov--Todorcevic correspondence in this section. We then establish the (stable) Ramsey property for the class of Banach spaces $\{\ell_\infty^n\}_n$. From this, we infer that that the group of isometries of the {\em Gurarij} space is extremely amenable.  


%

\subsection{Colorings and extreme amenability}\label{njhjiuyrtt}

We introduce some terminology to be used in the following. A \emph{metric
coloring} of a pseudo-metric space $M$ is a $1$-Lipschitz map from $M$ to a
metric space $(K,d_{K})$. A metric coloring with target space $\left(
K,d_{K}\right) $ will also be called a $K$-coloring. Following \cite%
{melleray_extremely_2014}, a \emph{continuous coloring} is a metric coloring
whose target space is the closed unit interval $\left[ 0,1\right] $. A \emph{%
compact coloring }is a metric coloring whose target space is a compact
metric space. For a subset $X$ of a compact metric space $\left(
K,d_{K}\right) $ and $\varepsilon >0$, the $\varepsilon $-fattening $%
K_{\varepsilon }$ denotes the set of elements of $K$ at distance at most $%
\varepsilon $ from some element of $X$.

The \emph{oscillation }$\mathrm{osc}_{F}(c)$ of a compact coloring $%
c:M\rightarrow (K,d_{K})$ on a subset $F$ of $M$ is the supremum of $%
d_{K}(c(y),c(y^{\prime }))$ where $y,y^{\prime }$ range within $F$. If $%
\mathrm{osc}_{F}(c)\leq \varepsilon $, then we  say that $c$ $%
\varepsilon $-\emph{stabilizes }on $F$, or that $F$ is $\varepsilon $-\emph{%
monochromatic }for $c$. A \emph{finite} coloring of $M$ is a function  from  $M$ to a finite set $X$. When the target space is a natural number $%
r$ (identified with the set $\left\{ 0,1,\ldots ,r-1\right\} $ of its
predecessors), we will say that $c$ is an $r$-coloring. A subset $F$ of $M$
is \emph{monochromatic} for $c$ if $c(p)=c(q)$ for every $p,q\in F$, and $%
\varepsilon $\emph{-monochromatic} for $c$ if there exists some $x\in X$
such that for every $p\in F$ there is $q\in M$ such that $c(q)=x$ and $%
d_{M}(p,q)\leq \varepsilon $. If $F$ is $\varepsilon $-monochromatic, then
we also say that $c$ $\varepsilon $-stabilizes on $F$.

Given a Polish group $G$ and a continuous action $G\curvearrowright M$ of $G$
on a metric space $(M,d_{M})$, we write $[p]_{G}$ to denote
the closure of the $G$-orbit of $p\in M$, and $M/\hspace{-0.1cm}/G$ to
denote the space of closures of $G$-orbits of $M$. Since $G$ acts by
isometries the formula 
\begin{equation*}
\widehat{d}_{G,M}([p],[q]):=\inf \{{d_{M}(\bar{p},\bar{q})}\,:\,{\bar{p}\in
[p],\,\bar{q}\in [q]}\}
\end{equation*}%
defines the quotient pseudometric induced by the quotient map $\pi
_{M,G}:M\mapsto M/\hspace{-0.1cm}/G$, and since we consider \emph{closures}
of orbits, $\widehat{d}_{G,M}$ is a metric. It is easy to see that $\widehat{%
d}_{G,M}$ is complete when $d_{M}$ is complete.

When $M$ is endowed with an action of a Polish group $G$ we  say that $M$ is a metric $G$-space. A compact coloring $c:(M,d_M)\to (K,d_K)$ is {\em finitely $G$-factorizable} when there is a $K$-coloring $\widehat{c}:M\quo G\to K$  defined on the space $M\quo G$ of closed $G$-orbits of $M$ such that for every $\vep>0$ and every  compact subset $F\con M$ there is some $g\in G$ such that $d_K(c(p), \widehat{c}([p]_G))\le \vep$ for every $p\in g\cdot F$, where $[p]_G$ is the closed $G$-orbit of $p$.  Similarly, $c$ is \emph{finitely oscillation stable }%
\cite[Definition 1.1.8]{pestov_dynamics_2006} if for every compact subset $F$
of $M$ and $\varepsilon >0$ there exists $g\in G$ such that $c$ \emph{$%
\varepsilon $-stabilizes} on $g\cdot F$. We say that the action of $G$ on $M$
is \emph{finitely oscillation stable }if every continuous coloring of $M$ is
finitely oscillation stable \cite[Definition 1.1.11]{pestov_dynamics_2006}.

%
%
%

Given a compact metric space $(K,d_{K})$, we let $\mathrm{Lip}%
((M,d_{M}),(K,d_{K}))$ be the collection of all $K$-colorings of $M$. With
the topology of pointwise convergence $\mathrm{Lip}((M,d_{M}),(K,d_{K}))$ is
a compact space, which is metrizable when $(M,d_{M})$ is separable. The
continuous action $G\curvearrowright (M,d_{M})$ induces a natural continuous
action $G\curvearrowright \mathrm{Lip}((M,d_{M}),(K,d_{K}))$, defined by
setting $(g\cdot c)(p):=c(g^{-1}\cdot p)$ for every $c\in \mathrm{Lip}%
((M,d_{M}),(K,d_{K}))$ and $p\in M$.

\begin{lemma}\label{Lemma:oscillation}

 Suppose that $G$ is a Polish group, and that $M$ is a metric $G$-space. Let   $\mc F$ be a $\con$-directed family of compact subsets of $M$ whose union is $M$.  The following assertions are equivalent:
\begin{enumerate}[1)]
\item  Every compact coloring of $M$ is finitely $G$-factorizable.
\item  For every $F\in\mc F$, every compact metric space $K$ and every $\vep>0$ there is a $H\in \mc F$ such that for every coloring $c:H\to K$  there is a coloring $\widehat{c}:H\quo G\to K$ and $g\in G$ such that $g F\con H$ and such that $d_K( c(p), \widehat{c}([p]_G))\le \vep$ for every $p\in g \cdot F$.
\end{enumerate}

\end{lemma}

\begin{proof}
Suppose that {\it 1)} holds but not {\it 2)}. Fix the counterexample $K$, $\mc F$, $M$, $\vep>0$ and $F\in \mc F$ and for each $H\in \mc F$ containing $F$  we fix a bad coloring $c_H: H\to K$.  For each $V\in \mc F$, let $\langle V\rangle$ be the collection of those $W\in\mc F$ containing $V$.    Choose a non-principal ultrafilter $\mc U$ on $\mc F$ containing each $\langle V\rangle$. This is possible since $\mc F$ is $\con$-directed. Define $c_\mc U: M\to K$ by declaring $c_\mc U(p):=\mc U-\lim c_H(p)$. This is well defined because there is some $H\in \mc F$ such that $p\in H$. Let $\widehat{c}:M\quo G \to K$ be the corresponding factorization, and let $g$ be such that $d_K(\widehat{c}([p]_G), c_\mc U(p))\le \vep/2$ for every $p\in g\cdot F$. Choose $H\in \mc F$ to  be such that $p\cdot F\con H$ and such that $d_K(c_H(p), c_\mc U(p))\le \vep/2$ for every $p\in g\cdot F$. Then the restriction $\widehat{c}: H\quo G\to K$ disproves that $c_H$ is a bad color.  Suppose now that {\it 2)} holds but not {\it 1)}. This means that there is some $c:M\to K$ that cannot be finitely $G$-factorized, so we fix the corresponding $\vep>0$.   For every  $F\in \mc F$ we use {\it 2)} for it,  $K$, and for $\vep/2$ to find the corresponding    $H_F\in \mc F$, and then we apply the property of it to the restriction $c: H\to K$ to find  $e_F: H\quo G\to K$.   Now define $\widehat{c}: M\quo G\to K$ as  the $\mc U$-limit of $(e_F)_{F}$.   Since $\widehat{c}$ does not finitely $G$-factorize $c$ there must be a bad compact $A$ witnessing this. Without loss of generality we may assume that $A$ is a finite set. Let $F\in \mc F$ be such that $A\con F$, and let $H\in \conj{V\in \langle F\rangle}{ d_K(\widehat{c}([p]_G), e_V([p]_G))\le \vep/2 \text{ for every $p\in F$}}\in \mc U$. Let $g\in G$ be such that 
$d_K(c(p), e_H([p]_G))\le \vep/2$ for every $p\in g \cdot F$, and consequently, $d_K(c(p),\widehat{c}([p]_G))\le \vep$ for every $p\in g\cdot A$, contradicting the defining property of $A$. 
\end{proof} 

 
Recall that a topological group $G$ is called \emph{extremely amenable} if
every continuous action of $G$ on a compact Hausdorff space has a fixed
point.  The following characterization of extreme amenability will be used
extensively in this paper. 

\begin{proposition}
\label{factor_orbitspace0} Suppose that $G$ is a Polish group. The following assertions are
equivalent.

\begin{enumerate}[1)]
\item $G$ is extremely amenable.

\item For every  
  left-invariant compatible metric $d_{G}$ on $G$, the left translation of $G$ on $\left( G,d_{G}\right) $ is finitely
oscillation stable.

\item Every compact coloring of a metric $G$-space is finitely $G$-factorizable.


\item Let $M$ be a metric $G$-space, and let $\mc F$ be a $\con$-directed family of compact subsets of $M$ whose union is $M$. For every $F\in\mc F$, every compact metric space $K$ and every $\vep>0$ there is an $H\in \mc F$ such that for every coloring $c:H\to K$  there is a coloring $\widehat{c}:H\quo G\to K$ and $g\in G$ such that $g F\con H$ and such that $d_K( c(p), \widehat{c}([p]_G))\le \vep$ for every $p\in g \cdot F$.

\end{enumerate}
\end{proposition}

\begin{proof}
The equivalence of \textit{1)} and \textit{2)} can be found in \cite[%
Theorem 2.1.11]{pestov_dynamics_2006}. The implication \textit{3)}$%
\Rightarrow $\textit{2)} is immediate, since orbit space $G/\hspace{-0.1cm}%
/G$ is one point. We now establish the implication \textit{1)}$\Rightarrow $%
\textit{3)}: Fix a 1-Lipschitz $c:(M,d_{M})\rightarrow (K,d_{K})$. Let $L$
be the closure of the $G$-orbit of $c$ in $\mathrm{Lip}((M,d_{M}),(K,d_{K}))$%
. By the extreme amenability of $G$, there is some $c_{\infty }\in L$ such
that $G\cdot c_{\infty }=\{c_{\infty }\}$, so we can define $\widehat{c}:M/%
\hspace{-0.1cm}/G\rightarrow K$ by $\widehat{c}([p]_{G}):=c_{\infty }(p)$.
It is clear that $\widehat{c}$ is 1-Lipschitz. Given a compact subset $F$ of 
$M$, let $g\in G$ be such that $\max_{p\in F}d_{K}(c_{\infty }(p),c(g\cdot
p))<\varepsilon $. If $x\in F$, then $d_{K}(c(g\cdot x),\widehat{c}([g\cdot
x]_{G}))=d_{K}(c(g\cdot x),c_{\infty }(x))<\varepsilon $. The equivalence of  \textit{(3)} and 
\textit{(4)} follows from Lemma \ref{Lemma:oscillation}.
\end{proof}

\subsection{The  Ramsey property and the KPT correspondence for Banach spaces}\label{oi9998555}

In this section, we provide a  characterization of extreme amenability of the isometry group of a Banach space (endowed with the topology of pointwise convergence).  This can be seen as an analogue in this context of the Kechris-Pestov-Todorcevic from \cite{kechris_fraisse_2005}. A more general KPT correspondence for arbitrary metric structures is the topic of \cite{melleray_extremely_2014}.

We introduce some basic terminology on Banach spaces.  Let $\mbb F$ be equal to $ \mbb R$ or to $\mbb C$. Given $n\in \N$, and $1\le p<\infty$, let $\ell_p^n$ be the normed space $(\mbb F^n,\nrm{\cdot}_p)$ where $\nrm{(a_j)_{j<n}}_p=(\sum_{j<n} |a_j|^p)^{1/p}$ is the $p$-norm; similarly, let $\ell_\infty^n=(\mbb F^n,\nrm{\cdot}_\infty)$ where $\nrm{(a_j)_{j<n}}_\infty:=\max_{j<n} |a_j|$. Given a Banach space $(X,\nrm{\cdot})$, let  $\mr{Ball}(X):=\conj{x\in X}{\nrm{x}\le 1}$, $\mr{Sph}(X):=\conj{x\in X}{\nrm{x}=1}$ be the unit ball an the unit sphere of $X$, respectively. Recall that given two Banach spaces $X,Y$,  a {\em contraction} $T: X\to Y$ is a bounded linear mapping $T:X\to Y$ such that $\nrm{T}:=\max_{\nrm{x}\le 1}\nrm{T(x)}\le 1$. Given $\de\ge 0$, let $\Emb_\de(X,Y)$ be the space of contractions $T:X\to Y$ such that $\nrm{Tx}\ge  \nrm{x}/(1+\de)$, endowed with its norm metric, $d(T,U):=\nrm{T-U}:=\max_{\nrm{x}\le 1}{\nrm{T(x)-U(x)}}$; when $\de=0$, $\Emb(X,Y):=\Emb_0(X,Y)$ is the space of {\em isometric embeddings} from $X$ into $Y$.   Dually, when $X$ and $Y$ are finite-dimensional, a {\em quotient map} $T:X\to Y$ is a linear mapping such that $T(\ball(X))=\ball(Y)$. The space of those quotient maps is denoted by $\Quo(X,Y)$. It is well-known that   $T\in \Emb(X,Y)$ if and only if its {\em dual} operator $T^*:Y^*\to X^*$ is a quotient map, and this assignment is an isometry.  Finally, given a   Banach space $E$,  let $\iso(E)$  be the group of surjective isometries of $E$, endowed with its strong operator topology (SOT), and  observe that  $\iso(E)$ acts  continuously on $\Emb_\de(X,E)$ by left composition, $g\cdot T:= g\circ T$.

%
%
%
%
%

In particular, suppose that  $X$ is a finite-dimensional subspace of $E$. Given a finite-dimensional subspace $Y$ of $E$ containing $X$ we can canonically identify $\Emb(X,Y)$ with the collection of those isometric embeddings $T:X\to E$ such that $\im T\con Y$, so in this way $\Emb(X,E)=\bigcup_{X\con Y\con E} \Emb(X,Y)$, where each $\Emb(X,Y)$ is a compact subset of $\Emb(X,E)$.  Suppose that $\iso(E)$ is extremely amenable. By applying Proposition \ref{factor_orbitspace0} we obtain that  given
such  $X\con Y$, compact metric $(K,d_K)$ and $\vep>0$ we can find   some  finite-dimensional subspace $Z_0$ of $E$  such that for every coloring $c:\Emb(X,Z_0)\to K$ there is  some $g\in \iso(E)$ such that 
\begin{equation}\label{lkejirwijree}
\text{there is a coloring $\widehat{c}:\Emb(X,Z_0)\quo \iso(E)\to K$ and $\max_{\ga\in \Emb(X,Y)}d_K( c(g\circ \ga) , \widehat{c}([\ga]_{\iso(E)}))<\vep/2$.}
\end{equation}
We consider on $\mr{Lip}(\Emb(X,Z_0), K)$  the compatible metric  defined  for $K$-colorings $c_1$ and $c_2$ by $d(c_1,c_2):=\max_{\ga\in \Emb(X,Z_0)}d_K(c_1(\ga),c_2(\ga))$. Using that   $\mr{Lip}(\Emb(X,Z_0),K)$ is compact, we can  find a finite $\vep/2$-dense subset $D$ of it,  and for each $c\in D$ we choose some $g_c\in \iso(E)$   witnessing \eqref{lkejirwijree}.  Let $Z$ be a  finite-dimensional subspace of $E$ containing $Y$ and $\bigcup_{c\in D}g_c Y$. Then $Z$ has the  property that for every coloring $c:\Emb(X,Z)\to K$ there is $g\in \iso(E)$ and $\widehat c: \Emb(X,Z)\quo \iso(E)\to  K $ with the property that $g Y\con Z$ and $d_K(c(g\circ \ga),\widehat c([g]_{\iso(E)}))\le \vep$ for every $\ga\in \Emb(X,Y)$.  This means in particular that the oscillation of $c$ in $g\circ \Emb(X,Y)$ is determined by the diameter of $\Emb(X,E)\quo \iso(E)$.  Recall that an action $G\acts M$ of a group $G$ on a metric space $ (M,d)$ is {\em $\vep$-transitive}, for some $\vep>0$, when the diameter of $M\quo G$ is at most $\vep$, that this, when  for every $x,y\in M$ there is some $g\in G$ such that $d(g \cdot x,y)\le \vep$.  $G\acts M$ is {\em approximately transitive} when it is $\vep$-transitive for every $\vep>0$, or equivalently, when $M\quo G$ consists of a point.   
\defi\label{approx_uh}
A Banach space $E$ is called {\em approximately ultrahomogeneous} when for every finite-dimensional subspace $X$ of $E$  one has that $\iso(E)\acts \Emb(X,E)$ is approximately transitive.   
\fdefi

Hence, we obtain the following. 
\cor\label{i3ij3ior3}
Suppose that $E$ is approximately ultrahomogeneous, and suppose that $\iso(E)$ is extremely amenable.   Then for every finite-dimensional subspaces $X\con Y$ of $E$ and every compact metric space $(K,d_K)$ there is some finite-dimensional subspace $Z$ of $E$ containing $Y$ with the property that every coloring $c:\Emb(X,Z)\to K$ $\vep$-stabilizes in some set of the form $\ga\circ \Emb(X,Y)$.  \qed
\fcor	

Up to now the list of known approximately ultrahomogeneous (real or complex) Banach spaces is: 
\begin{enumerate}[$\bullet$]
\item  Hilbert spaces (indeed, they are ultrahomogeneous, i.e. the algebraic quotients $\Emb(X,E)/G$ consist always of a point);
\item The Lebesgue spaces  $L_p[0,1]$ when $p\notin 2\N$, proved by W. Lusky in \cite{lusky_consequences_1978};
\item The Gurarij space $\mathbb{G}$.

\end{enumerate} 
The original
characterization of the Gurarij space considered by Gurarij \cite%
{gurarij_spaces_1966} and Lusky \cite%
{lusky_gurarij_1976,lusky_separable_1977,lusky_construction_1979} is as
the unique separable Banach space satisfying the following extension
property: for every finite-dimensional\ Banach spaces $E\subseteq F$, linear
contraction $\vphi :E\rightarrow \mathbb{G}$, and $\varepsilon >0$, there
exists an extension $\hat{\vphi}:F\rightarrow \mathbb{G}$ satisfying $||\hat{%
\vphi}||{}<1+\varepsilon $. The fact that such a space is indeed approximately ultrahomogeneous as in
Definition \ref{approx_uh} is proved by I.\ Ben Yaacov in \cite{ben_yaacov_fraisse_2015}.

The isometry groups (endowed with the strong operator topology) of the Banach spaces in the list above have very special topological dynamical properties. The groups $\iso( L_p(0,1))$ are extremely amenable for every $1\le p<\infty$, which was proved in the case of $p=2$ by  M. Gromov and V. D. Milman \cite{gromov_topological_1983} and for $p\neq 2$ by T. Giordano and V. Pestov \cite{giordano_extremely_2007}. Both of the cases use the method of concentration of measure. In this paper we prove the following.
\begin{theorem}\label{zcxvbnjhgfd}
 The group of isometries of the Gurarij space endowed with the strong operator topology is extremely amenable. 
\end{theorem}  
Our proof is not based on concentration of measure, but on a combinatorial property, the approximate Ramsey property, that characterizes the     extreme amenability of certain  isometry groups.  With a similar approach this has been extended in \cite{bartosova_ramsey_2017}     to the context of operator spaces.   
We now to introduce several variants of the Ramsey property for Banach spaces.

\begin{definition}[Approximate  Ramsey Property]
\label{Definition:ARPdfsfds}  Let  $\mc F$   be a family of finite-dimensional Banach spaces.
 \begin{enumerate}[a)]
\item  $\mc F$   satisfies the \emph{approximate Ramsey property (ARP)}
if for any $X,Y\in \mathcal{F}$ and $\varepsilon >0$  there exists $Z\in 
\mathcal{F}$ such that any  continuous coloring  of $\mathrm{Emb}
(X,Z)$ $\varepsilon $-stabilizes on $\gamma \circ \mathrm{Emb}%
(X,Y)$ for some $\gamma \in \mathrm{Emb}(Y,Z)$.

\item   $\mc F$ satisfies the \emph{compact approximate Ramsey property}
when for any $X,Y\in \mathcal{F}$, $\varepsilon >0$ and compact metric space $(K,d_K)$ there exists $Z\in 
\mathcal{F}$ such that any $K$-coloring  of $\mathrm{Emb}
(X,Z)$ $\varepsilon $-stabilizes on $\gamma \circ \mathrm{Emb}%
(X,Y)$ for some $\gamma \in \mathrm{Emb}(Y,Z)$.  
  \item  $\mathcal{F}$ satisfies the \emph{discrete approximate Ramsey property}   when for every $X,Y\in \mc F$, $r\in \N$ and $\vep>0$ there is some $Z\in \mc F$ such that any $r$-coloring of $\Emb(X,Z)$ $\vep$-stabilizes on $\ga\circ \Emb(X,Y)$ for some $\ga\in \Emb(Y,Z)$.

\end{enumerate}  

%
%
%
\end{definition}

So, rephrasing Corollary \ref{i3ij3ior3}, if $E$ is an approximately ultrahomogeneous Banach space whose isometry group is extremely amenable, then the class $\mr{Age}(E)$ of finite-dimensional subspaces of $E$ has the approximate Ramsey property.   Conversely, we will see in Theorem   \ref{liwjr3iwejirwe} that in fact the (ARP) of $\age(E)$ characterizes the extreme amenability of $\iso(E)$ for approximately ultrahomogeneous spaces $E$.   Now we show that the different versions of the  Ramsey property are in  fact equivalent. 

\begin{proposition}
\label{ARP=DARP} The following are equivalent for a class $\mathcal{F}$ of finite-dimensional Banach spaces:

\begin{enumerate}[1)]
\item $\mathcal{F}$ satisfies the (ARP).

\item $\mathcal{F}$ satisfies the compact (ARP). 
\item $\mathcal{F}$ satisfies the \emph{discrete} (ARP).

\end{enumerate}
\end{proposition}

\begin{proof}
The  compact (ARP) obviously implies the   (ARP). Suppose that $\mathcal{F}$ satisfies the   (ARP), and let us prove that $%
\mathcal{F}$ satisfies the discrete (ARP). This is done by induction on $%
r\in \mathbb{N}$. The case $r=1$ is trivial. Suppose that we have shown that 
$\mathcal{F}$ satisfies the discrete (ARP) for $r$-colorings. Consider $%
X,Y\in \mathcal{F}$ and $\varepsilon >0$. Then by the inductive hypothesis,
there is $Z_{0}\in \mathcal{F}$ such that every $r$-coloring of $\mathrm{%
\mathrm{Emb}}\left( X,Z_{0}\right) $ $\varepsilon $-stabilizes
on $\gamma \circ \mathrm{Emb}\left( X,Y\right) $ for some $%
\gamma \in \mathrm{Emb}\left( Y,Z_{0}\right) $. Since by the
assumption $\mathcal{F}$ satisfies the continuous (ARP), there is $Z\in \mathcal{F}$
such that every continuous coloring of $\mathrm{Emb}\left(
X,Z\right) $ $\varepsilon /2$-stabilizes on $\gamma \circ \mathrm{Emb}\left( X,Z_{0}\right) $ for some $\gamma \in \mathrm{Emb}\left( Z_{0},Z\right) $. We claim that $Z$ witnesses that $%
\mathcal{F}$ satisfies the discrete (ARP) for $\left( r+1\right) $%
-colorings. Indeed, suppose that $c$ is an $\left( r+1\right) $-coloring of $%
\mathrm{Emb}\left( X,Z\right) $. Define $f:\mathrm{Emb}\left( X,Z\right) \rightarrow \left[ 0,1\right] $ by $f\left(
\phi \right) :=\frac{1}{2}d\left( \phi ,c^{-1}\left( r\right)
\right) $. This is a continuous coloring, so by the choice of $Z$ there
exists $\gamma \in \mathrm{Emb}(Z_{0},Z)$ such that $f$ $%
\varepsilon /2$-stabilizes on $\gamma \circ \mathrm{Emb}(X,Z_{0})$. Now, if there is some $\phi \in \mathrm{Emb}(X,Z_{0})$ such that $c(\gamma \circ \phi )=r$, then $\gamma \circ \mathrm{%
Emb}(X,Z_{0})\subseteq (c^{-1}(r))_{\varepsilon }$, so choosing
an arbitrary $\bar{\gamma}\in \mathrm{Emb}(Y,Z_{0})$ we obtain
that $c$ $\varepsilon $-stabilizes on $\gamma \circ \bar{\gamma}\circ 
\mathrm{Emb}(X,Y)$. Otherwise, $(\gamma \circ \mathrm{Emb}(X,Z_{0}))\cap c^{-1}(r)=\emptyset $, so defining $\bar{c}(\phi
):=c\left( \gamma \circ \phi \right) $ for $\phi \in \mathrm{Emb}(X,Z_{0})$ gives an $r$-coloring of $\mathrm{Emb}(X,Z_{0})$.
By the choice of $Z_{0}$ there exists $\bar{\gamma}\in \mathrm{Emb}(Y,Z_{0})$ such that $\bar{c}$ $\varepsilon $-stabilizes on $\bar{\gamma}%
\circ \mathrm{Emb}(X,Y)$. Therefore $c$ $\varepsilon $%
-stabilizes on $\gamma \circ \bar{\gamma}\circ \mathrm{Emb}(X,Y) $. This concludes the proof that the continuous (ARP) implies the discrete (ARP).

Finally, the discrete (ARP) implies the  compact (ARP). In fact, given $%
\varepsilon >0$ and a compact metric space $K$, one can find a finite $%
\varepsilon $-dense subset $D\subseteq K$. Thus if $Z\in \mathcal{F}$
witnesses the discrete (ARP) for $X,Y$, $\varepsilon $ and $D$, then given a 
$1$-Lipschitz $f:\mathrm{Emb}(X,Z)\rightarrow K$ we can define a
coloring $c:\mathrm{Emb}(X,Z)\rightarrow D\subseteq K$ such
that $d_{K}(c(\phi ),f(\phi ))\leq \varepsilon $ for every $\phi \in \mathrm{%
Emb}(X,Z)$. In this way, if $c$ $\varepsilon $-stabilizes on $%
\gamma \circ \mathrm{Emb}(X,Y)$, then $f$ $3\varepsilon $%
-stabilizes on $\gamma \circ \mathrm{Emb}(X,Y)$.
\end{proof}

%

We are going to see that, when $E$ is approximately ultrahomogeneous, the extreme amenability of $\iso(E)$ is equivalent to the (ARP) of $\age(E)$ and, in fact, also to a stronger version of the Ramsey property for a rich subfamily of $\age(E)$. To state this property we  
recall that for two $k$-dimensional Banach spaces $X, Y$, the {\em Banach-Mazur} (pseudo)distance is defined by
$$d_{\mr{BM}}(X,Y)=\log \left( \min\conj{\nrm{T}\cdot \nrm{T^{-1}}}{T:X\to Y \text{ is a linear isomorphism}}  \right).$$ 
\defi
Given a family $\mc F$ of finite-dimensional Banach spaces, let $[\mc F]$ be the class of all separable Banach spaces $E$ such that $\mc F\con \mr{Age}(E)$, and such that  every finite-dimensional subspace of $E$ is the $d_{\mr{BM}}$-limit of a sequence of subspaces of elements of $\mc F$. 
\fdefi	
For example, the spaces $c_0$,  $C[0,1]$ or the Gurarij space are in the class $[\{\ell_\infty^n\}_n]$, where each $\ell_\infty^n$ is the (real or complex) vector space $\mbb F^n$ endowed with the sup norm, $\nrm{(a_1,\dots,a_n)}_\infty:=\max_{i} |a_i|$.  In general $[\{\ell_\infty^n\}_n]$ is the class of separable  Lindenstrauss spaces.   In the next, by a {\em modulus of stability} we mean a function $\varpi:[0,\infty[\to [0,\infty[$ that is increasing and continuous at zero with value zero.

\defi[Fraïssé properties]
Let $E$ be a separable Banach space, and let $\mc F$ be a family of finite-dimensional spaces. 
\begin{enumerate}[a)]
\item  $E$
satisfies the stable homogeneity property with respect to $\mathcal{F}$ with modulus of stability $%
\varpi $ if $\Emb(X,E)$ is nonempty for every $X\in \mathcal{F}$ and if for every $X\in \mathcal{F}$, $\delta \geq 0$, $\varepsilon >0$,  the canonical action $\iso(E)\acts \Emb_\de(X,E)$ is $(\varpi(\de)+\vep)$-transitive.

\item  $E$ is a {\em stable Fraïssé} Banach space with modulus of stability $\varpi$ when $E$ satisfies the stable homogeneity property with respect to $\age(E)$. 
\item $\mc F$   satisfies the {\em stable amalgamation property (SAP)} with modulus $\varpi$ when for every   $X,Y,Z\in \mc F$, $\vep>0$, $\de\ge 0$,  $\ga\in \Emb_\de(X,Y)$ and $\eta\in \Emb_\de(X,Z)$  there is $V\in \mc F$, $I\in \Emb(Y,V)$ and $J\in \Emb(Z,V)$ such that $\nrm{I\circ \ga-J\circ \eta}\le \varpi(\de)+\vep$.
  \item $\mc F$  is a {\em stable Fraïssé class} when $\mc F$ satisfies the (SAP) and the {\em joint embedding property (JEP)}, that is, for every $X,Y\in \mc F$ there is $Z\in \mc F$ such that $\Emb(X,Z),\Emb(Y,Z)$ are nonempty. 
\end{enumerate}  
\fdefi
It is easy to see that if $\mc F$ satisfies the (SAP) and it has a least element with respect to inclusion, then $\mc F$ has the (JEP). Using the fact that $\{\ell_\infty^n\}_n$ is a stable Fraïssé class with modulus $\varpi(\de)=\de$ (see Proposition \ref{li23jior3ji3r}),   it is  proved in \cite[\S\S 6.1]{lupini_fraisse_2016}   that the Gurarij space is a stable Fraïssé Banach space with modulus $\varpi(\de)=\de$.  In fact, this approximate ultrahomogeneity is a direct consequence of the fact that the Gurarij space is the ``generic'' direct limit of the class of all finite-dimensional Banach spaces, an instance of the following Fraïssé correspondence for Banach spaces (see for instance \cite[%
Subsection 2.6]{lupini_fraisse_2016}).

\begin{proposition}
\label{Proposition:exists}Suppose that $\mathcal{F}$ is a class of finite-dimensional Banach spaces,  and $E$ is a separable Banach space.  Then,
\begin{enumerate}[a)]
\item If $E$ is a Fraïssé   space with modulus $\varpi$, then $\age(E)$ is a stable Fraïssé class with modulus $\varpi$.
\item    If $\mc F$ is
a stable Fra\"{\i}ss\'{e} class with modulus $\varpi$, then
there is a unique  separable $E\in [\mc F]$ that satisfies the stable homogeneity property with respect to $\mc F$ with modulus $\varpi$. This space is called the {\em Fra\"{\i}ss\'{e} limit of $\mc F$} and denoted by $\flim\mathcal{F}$. \qed
  
\end{enumerate}  
%
%
%
%
%
%
%
\end{proposition}
Consequently, the class of all finite-dimensional Banach spaces is stable with modulus $\de$. The classes  $\{\ell_p^n\}_n$ for $1\le p\le \infty$ are also stable:  The case $p=\infty$ is rather easy to prove (see  Proposition \ref{li23jior3ji3r}), as well as the  case $p=2$, where one  can use the polar decomposition; for $1<p<\infty$, $p\neq 2$, one can use a result  of G. Schechtman in \cite{Schechtman_1979} of approximation of $\de$-embeddings by isometric embeddings. Also, it is proved in \cite{ferenczi_amalgamation_2019} that  for $p\ne 4,6,8,\dots$, the class $\age(L_p(0,1))$ has a weaker form of stable approximate ultrahomogeneity, namely one that may depend on the dimension.   Several other examples of Fraïssé classes of structures in functional analysis are studied in  \cite{lupini_fraisse_2016}.

As we mentioned before, we will see that for an approximately ultrahomogeneous space $E$, the (ARP) of its age   is equivalent to  the extreme amenability of the isometry group of $E$. Furthermore, when $E=[\mc F]$ for some stable Fraïssé class $\mc F$,  a stronger form of the (ARP) of $\mc F$ is also equivalent to the   extreme amenability of the isometry group of $E$.

\begin{definition}
\label{Definition:ARP}  A class  $\mathcal{F}$ of   finite-dimensional Banach spaces 
  satisfies the \emph{stable approximate Ramsey property (SRP)} with
stability modulus $\varpi $ if for any $X,Y\in \mathcal{F}$, $\varepsilon >0$%
, $\delta \geq 0$ there exists $Z\in \mathcal{F}$ such that every 1-Lipschitz mapping
 $c:\mathrm{\ \mathrm{Emb}}_{\delta }(X,Z) \to [0, 2(1+\de)]$  $%
(\varpi (\delta )+\vep)$-stabilizes on $\gamma \circ \mathrm{Emb}%
_{\delta }(X,Y)$ for some $\gamma \in \mathrm{Emb}(Y,Z)$.

The    \emph{compact }(SRP) and  \emph{discrete }(SRP) are defined as the
(ARP),  by replacing continuous colorings with compact and  
finite colorings, respectively.
%
\end{definition}

%

%


\begin{theorem}[KPT correspondence for Banach spaces]\label{liwjr3iwejirwe} Let $E$ be an approximately ultrahomogeneous Banach  space. Then the following are equivalent:
\begin{enumerate}[1)]
\item $\iso(E)$ is extremely amenable.
\item $\age(E)$  satisfies the approximate Ramsey property.
\item For every $X,Y\in \age(E)$, every $\vep>0$  and every continuous coloring $c$ of $\Emb(X,E)$ there is some $g\in \iso(E)$ such that $\osc(c\rest g\circ \Emb(X,Y))\le \vep$.
\end{enumerate}
If in addition $\mc F$ is  a family that satisfies the stable amalgamation property such that $E\in [\mc F]$ and  $\mc F\preceq \age(E)$, that is, every space in $\mc F$ can be isometrically embedded into $E$, then   1), 2), 3) above are also  equivalent to 
\begin{enumerate}[1)]\addtocounter{enumi}{3}
\item   $\mc F$ satisfies the (SRP).
  \end{enumerate}

\end{theorem}

The equivalence of {\it1)} and {\it 2)}   is a particular instance a   more general characterization of extreme amenability in terms of an approximate Ramsey property when Banach spaces are  regarded as {\em metric structures}  \cite{ben_yaacov_model_2008} as in \cite[%
Appendix B]{goldbring_kirchbergs_2015} or \cite[\S 8.1]{lupini_fraisse_2016}.  Before we give a proof of the correspondence, we compare these Ramsey properties.

\begin{proposition}
\label{SRP=cSRP} Suppose that $\mathcal{F}$ is a class of finite-dimensional spaces with the {\em joint embedding embedding property}, that is, for every $X,Y\in \mc F$ there is $Z\in \mc F$ such that $\Emb(X,Z),\Emb(Y,Z)\neq \buit$. 
Then the following
assertions are equivalent:

\begin{enumerate}[1)]
\item $\mathcal{F}$ satisfies the (ARP) and the (SAP) with modulus $\varpi$.

\item $\mathcal{F}$ satisfies the (SRP) with modulus $\varpi $.

\item $\mathcal{F}$ satisfies the discrete (SRP) with modulus $\varpi $.

\item $\mathcal{F}$ satisfies the compact (SRP) with modulus $\varpi $.
\end{enumerate}
\end{proposition}

\begin{proof}
Trivially, the compact (SRP) with modulus $\varpi $ implies the discrete (SRP) with modulus 
$\varpi $, and a  simple modification of the proof of the Proposition \ref{ARP=DARP} gives
that the discrete (SRP) with modulus 
$\varpi $ implies the   (SRP) with modulus $\varpi$.   
Trivially, the   (SRP) with modulus $\varpi $ implies the  
(ARP). 
In addition, we have the following
\clam 
If $\mc F$ has the   (SRP) with modulus $\varpi$ then $\mc F$ has the (SAP) with  modulus $\varpi$.
\fclam	
\prucl	
Fix $X,Y,Z\in \mc F$, $\vep>0$, $\de\ge 0$ and $\ga\in \Emb_\de(X,Y)$ and $\eta\in \Emb_\de(X,Z)$. Find  $V\in \mc F$ such that $\Emb(X,V),\Emb(Y,V)$ and $\Emb(Z,V)$ are non empty. Find $W\in \mc F$ witnessing the (SRP) for initial parameters $X,V\in \mc F$, $\vep,\de$.     We claim that $W$ also witnesses the (SAP) for $\ga, \eta$, $\vep$ and $\de$.  Choose $\theta_Y\in \Emb(Y,V)$ and $\theta_Z\in \Emb(Z,V)$.  Let $I\in \Emb(V,W)$ be such that $\osc(c\rest I\circ \Emb_\de(X,V))\le \varpi(\de)+\vep$, where   $c:\Emb_\de(X,W)\to [0, 2+\de]$ is defined by  $c(\xi):=d(\xi,\Emb(V,W)\circ \theta_Y \circ \ga)$. Since $c(I\circ \theta_Y\circ \ga)=0$, $d(I\circ \theta_Z\circ \eta, \Emb(V,W)\circ \theta_Y\circ \ga)\le \varpi(\de)+\vep$, so there is some $J\in \Emb(V,W)$ such that $\nrm{I\circ \theta_Z\circ \eta - J \circ \theta_Y\circ \ga}\le \varpi(\de)+\vep$, as desired. 
\fprucl
Suppose   that $\mc F$ has the (ARP) and the (SAP) with modulus $\varpi$, and we prove that $\mc F$ has the compact (SRP). The next  claim is not difficult to prove.
\clam	\label{loikewijriowjrejwre}
$\mc F$ has the (SAP) with modulus $\varpi$ if and only if for every $X,Y\in \mc F$, $\de\ge 0$ and $\vep>0$ 
there exist $Z\in  \mathcal{F}$ and $I\in \mathrm{Emb}(Y,Z)$ such that
for every $\phi ,\psi \in \mathrm{Emb}_{\delta }(X,Y)$ there is $J\in 
\mathrm{Emb}(Y,Z)$ such that $\Vert I\circ \phi -J\circ \psi \Vert \leq \varpi (\delta )+\vep$. \qed

\fclam

 Fix $X,Y\in 
\mathcal{F}$, $\delta ,\varepsilon >0$ and a compact metric space $K$. We
use the previous claim to find $Y_{0}\in \mathcal{F}$
such that for every $\phi ,\psi \in \mathrm{Emb}_{\delta }(X,Y)$
there are $i,j\in \mathrm{Emb}(Y,Y_{0})$ such that $\Vert
i\circ \phi -j\circ \psi \Vert _{\mathrm{cb}}\leq\varpi
(\delta ) + \varepsilon$.
We consider the space $\mc L:=\mathrm{Lip}(\mathrm{Emb}_{\delta }(X,Y),K)$ of $1$-Lipschitz maps from $\mathrm{Emb}_{\delta }(X,Y)$ to $K$ as a compact metric space, endowed with the metric $d\left(
f,g\right) =\sup \left\{ d_{K}\left( f\left( \phi \right) ,g\left( \phi
\right) \right) :\phi \in \mathrm{Emb}_{\delta }(X,Y)\right\} $%
. By Proposition \ref{ARP=DARP}, $\mathcal{F}$ satisfies the compact (ARP).
Thus there exists some $Z\in \mathcal{F}$ such that every $\mc L$-coloring of $\mathrm{Emb}(Y,Z)$ $\varepsilon $-stabilizes on $\gamma \circ \mathrm{Emb}(Y,Y_{0})$ for some $\gamma \in \mathrm{Emb}(Y_{0},Z) $. We claim that $Z$ works, so let $c:\mathrm{Emb}_{\delta }(X,Z)\rightarrow K$ be $1$-Lipschitz. We can   define a coloring $\widehat{%
c}:\mathrm{Emb}(Y,Z)\to 
\mc L$ by setting, for $\xi \in \mathrm{%
Emb}(Y,Z)$, $\widehat{c}(\xi):\mathrm{Emb}_{\delta }(X,Y)\rightarrow K$, $\phi \mapsto c\left( \xi \circ \phi
\right) $. By the choice of $Z$, there exists $\bar{\gamma}\in \mathrm{Emb}(Y_{0},Z)$ be such that $\widehat{c}$ $\varepsilon $-stabilizes
on $\bar{\gamma}\circ \mathrm{Emb}(Y,Y_{0})$. Choose an
arbitrary $\varrho \in \mathrm{Emb}(Y,Y_{0})$. We claim that $c$
$\left( \varpi \left( \delta \right)+3\varepsilon \right) $-stabilizes on $%
\bar{\gamma}\circ \varrho \circ \mathrm{Emb}_{\delta }(X,Y)$.
Let $\phi ,\psi \in \mathrm{Emb}_{\delta }(X,Y)$. By the choice
of $Y_{0}$ there are $i,j\in \mathrm{Emb}(Y,Y_{0})$ such that $%
\nrm{ i\circ \phi -j\circ \psi } \leq \varpi (\delta )+\varepsilon $. Since $d_{\mc L}(\widehat{c}(\bar{\gamma}\circ \varrho ),%
\widehat{c}(\bar{\gamma}\circ i))\leq \varepsilon $ and $d_{\mc L}(\widehat{c}(%
\bar{\gamma}\circ \varrho ),\widehat{c}(\bar{\gamma}\circ j))\leq
\varepsilon $, it follows that $d_{K}(c(\bar{\gamma}\circ \varrho \circ \phi
),c(\bar{\gamma}\circ i\circ \phi ))\leq \varepsilon $, $d_{K}(c(\bar{\gamma}%
\circ \varrho \circ \psi ),c(\bar{\gamma}\circ j\circ \psi ))\leq
\varepsilon $. Furthermore, from $\nrm{ i\circ \phi -j\circ
\psi } \leq \varpi (\delta )+\varepsilon $ and the fact that $c$ is $1$%
-Lipschitz we deduce that  $d_{K}(c(\bar{%
\gamma}\circ \varrho \circ \phi ),c(\bar{\gamma}\circ \varrho \circ \psi
))\leq \varpi (\delta )+3\varepsilon $.
\end{proof}

\prue[Proof of Theorem \ref{liwjr3iwejirwe}] 
Corollary \ref{i3ij3ior3} gives that {\it 1)} implies {\it 2)}.  Let us prove the reverse. Suppose that $\age(E)$ has the (ARP). Let $(X_n)_n$ be an increasing sequence of finite-dimensional subspaces of $E$ whose union is dense in $E$, and let $d$ be the   metric on $\iso(E)$ defined by $d(g,h):=\sum_n 2^{-n-1} \nrm{g\rest X_n-h\rest X_n}$. Observe that $d$ is a left-invariant compatible metric on $\iso(E)$.  In order to prove the extreme amenability of $\iso(E)$ we  prove {\it 2)} in Proposition \ref{factor_orbitspace0} for the distance $d$, that is, that the left translation of $\iso(E)$ on $(\iso(E),d)$ is finitely oscillation stable. We fix a 1-Lipschitz mapping $c:\iso(E)\to [0,1]$,  a finite subset $F\con \iso(E)$ and $\vep>0$. Let $n$ be such that $2^{n-2} \vep\ge 1$ and let  $Y\con E$ be a finite-dimensional subspace of $E$ such that $X_n\cup \bigcup_{g\in F} g(X_n)\con Y$. Let $Y\con Z\con E$ be a finite-dimensional space witnessing the   (ARP) of $\mr{Age}(E)$ for the parameters $X_n, Y$ and $\vep/8$.   For each $\ga\in \Emb(X_n,Z)$ we choose $g_\ga\in \iso(E)$ such that $\nrm{\ga-g_\ga\rest X_n}\le \vep/8$, and now we define the (discrete) coloring $\widehat c:\Emb(X_n,Z)\to \{1,\dots,2^{n+1}\}$, by   $\widehat c(\ga)= j$ when  $j$ is the first integer $i$ such that $ c(g_\ga)\in J_i $, where $J_i:=[(i-1)/2^{n+1}, i/2^{n+1}]$.    There is some $\xi\in \Emb(Y,Z)$  and $ j\in \{1,\dots,2^{n+1}\}$ such that $\xi\circ \Emb(X_n,Y)\con (\widehat c^{-1}(j))_{\vep/8}$.  Choose $h\in \iso(E)$ such that $\nrm{\xi-h\rest Y}\le \vep/16$.  We claim that $\osc (c \rest h\cdot F)\le \vep$:  given  $g_0,g_1\in F$, there are $f_0,{f_1}\in \iso(E)$ such that $(j-1)/2^{n+1}\le c({f_0}), c({f_1})\le j/2^{n+1}$ and such that $\nrm{\xi\circ g_0\rest X_n - f_0\rest X_n},\nrm{\xi\circ g_1\rest X_n - f_1\rest X_n} \le \vep/4 $.  Hence $d(h\circ g_0,f_0),d(h\circ g_1,f_1)\le 7\vep/16$, and since $c$ is 1-Lipschitz, 
$$|c(h\circ g_0)- c(h\circ g_1)|\le d(h\circ g_0, f_0)+ |c(f_0)-c(f_1)| + d(h\circ g_1, f_1) \le \vep.$$
{\it 2)} and {\it 3)} are equivalent by Claim \ref{loikewijriowjrejwre}, under the hypothesis that $E$ is approximately ultrahomogeneous. 

 Suppose that   $\mc F$  is a family such that $\mc F\preceq \age(E)$, $E\in [\mc F]$  and suppose that it satisfies the   stable  amalgamation property.  We suppose first that {\it 2)} holds, that is,   $\age(E)$ has the (ARP), and we prove {\it 4)}: By Proposition   \ref{SRP=cSRP}  and Proposition 
    \ref{ARP=DARP},
it suffices to show that $\mathcal{F}$ satisfies the discrete (ARP). Fix $X,Y$ in $\mathcal{F}$, $r\in \mathbb{N}$, and $\varepsilon >0$. We know
by the hypothesis and Proposition \ref{ARP=DARP} that $\age(E)$ satisfies the discrete (ARP). Thus, we can find  $
Z_{0}\in \age(E)$ containing a copy of $Y$ and such that
every $r$-coloring of $\Emb(X,Z_{0})$ has an $\varepsilon $%
-monochromatic subset of the form $\gamma \circ \Emb(X,Y)$ for
some $\gamma \in \Emb(Y,Z_{0})$. Let $\de\le \vep$ be such that $\varpi(\de)<\vep$.  Let $Z_{1}\in \mc F$ for
which there exists an $\de $-embedding $\theta :Z_{0}\rightarrow
Z_{1}$. By the (SAP) of $ \mathcal{F}$ we can find $Z\in  \mathcal{F}$ and $%
I\in \Emb(Z_{1},Z)$ such that for every $\phi \in \Emb_{\de } (X,Z_{1})$ there is $\bar{\phi}\in \Emb(X,Z)$ such
that $\nrm{I\circ \phi -\bar{\phi}}\leq \varepsilon $, and
similarly for the elements of $\mathrm{Emb}_{\de }(Y,Z_{1})$. We claim that $Z$ witnesses the discrete (ARP) for the given $%
X,Y,\varepsilon ,r$. Fix a coloring $c:\Emb(X,Z)\rightarrow r$.
Define $b:\Emb(X,Z_{0})\rightarrow r$, by choosing for each $%
\phi \in \Emb(X,Z_{0})$ an element $\bar{\phi}\in \Emb(X,Z)$ such that $\nrm{I\circ \theta \circ \phi -\bar{\phi}}\leq \varepsilon $ and declaring $b(\phi ):=c(\bar{\phi})$. By the choice
of $Z_{0}$ from the discrete (ARP) of $\age(E)
$, there exist $\alpha \in \Emb(Y,Z_{0})$ and $j<r$ such that $%
\alpha \circ \Emb(X,Y)\subseteq (b^{-1}(j))_{\varepsilon }$.
Let $\bar{\alpha}\in \Emb(Y,Z)$ be such that $\nrm{
I\circ \theta \circ \alpha -\bar{\alpha}}\leq \varepsilon $. We claim that 
$\bar{\alpha}\circ \Emb(X,Y)\subseteq (c^{-1}(j))_{3\varepsilon
} $: Fix $\phi \in \Emb(X,Y)$. Let $\sigma \in \Emb(X,Z_{0})$ be such that $b(\sigma )=j$ and $d_{\mathrm{cb}}(\alpha \circ
\phi ,\sigma )\leq \varepsilon $. By definition, we can find $\bar{\sigma}%
\in \Emb(X,Z)$ such that $c(\bar{\sigma})=j$ and such that $
\nrm{I\circ \theta \circ \sigma -\bar{\sigma}}\leq \varepsilon $.
Then, 
\begin{equation*}
\nrm{\bar{\alpha}\circ \phi -\bar{\sigma}}\leq \nrm{
\bar{\alpha}\circ \phi -I\circ \theta \circ \alpha \circ \phi }+ \nrm{I\circ \theta \circ \alpha \circ \phi -I\circ \theta \circ \sigma }+\nrm{ I\circ \theta \circ \sigma -\bar{\sigma}}\leq 3\varepsilon .
\end{equation*}
Finally, suppose that {\it 4)} holds, that is, $\mc F$ has the stable approximate Ramsey property with modulus $\varpi$, and let us prove {\it 3)}: Let $\mc F_E$ be the collection of subspaces of $E$ that are isometric to some element of $\mc F$. Obviously, $\mc F_E$ also has the (ARP).  Fix $X,Y\in \age(E)$ and $\vep>0$.  We consider $0<\de\le 1$ such that $\varpi(\de)<\vep$  and $X_0\in \mc F_E$ such that there is $\theta\in \Emb_\de(X,X_0)$. Choose also a finite $\vep$-dense subset $D$ of $\Emb(X,Y)$, and for each $\ga\in D$ some $g_\ga\in \iso(E)$ such that $\nrm{g_\ga\rest X- \ga}\le \vep$. Let now $X_1\in \mc F_E$ be such that for every $\ga\in D$ there is $\eta\in \Emb_\de(X_0,X_1)$ such that $\nrm{g_\ga\rest X_0-\eta}\le \vep$. Let $Y_0\in \mc F_E$ and $\iota\in \Emb(X_1,Y_0)$ be such that  $\iota\circ \Emb_\de(X_0,X_1)\con (\Emb(X_0, Y_0))_\vep$. We use now the (ARP) of $\mc F_E$ when applied to $X_0,Y_0$ and $\vep/2$ to find  the corresponding $Z\in \mc F_E$. Fix a continuous coloring $c:\Emb(X,E)\to [0,1]$, and we define a continuous coloring $e:\Emb(X_0,Z)\to [0,1]$ as follows: Fix a non-principal ultrafilter $\mc U$ on $\N$. Given $\ga\in \Emb(X_0,Z)$ we  choose a sequence  $(g_n)_n$ in  $\iso(E)$ such that $\nrm{g_n \rest X_0 - \ga}\le 1/2^{n}$. Let $e(\ga):=\mc U-\lim (c(g_n\rest X))$. It is easy to see that $e$ is $(1+\de)$-Lipschitz. There is some $\ga\in \Emb(Y_0,Z)$ such that $\osc((e/(1+\de))\rest \ga\circ \Emb_\de(X_0,Y_0))\le \vep/2 $, hence $\osc(e\rest \ga\circ \Emb_\de(X_0,Y_0))\le \vep $. Let $h\in \iso(E)$ be such that $\nrm{h\rest X_1-\ga\circ \iota}\le  \vep$. We claim that $\osc(h  \circ \Emb(X,Y))\le 23\vep$:   Fix $\ga_0,\ga_1\in D$. Then, $\nrm{g_{\ga_j} \rest X- \ga_j}\le \vep$ for $j=0,1$. Choose $\eta_0,\eta_1\in \Emb_\de(X_0,X_1)$  such that $\nrm{g_{\ga_j}\rest X_0 -\eta_j}\le \vep$ for $j=0,1$. Choose $\xi_0,\xi_1\in \Emb(X_0,Y_0)$ such that $\nrm{\xi_j - \iota \circ \eta_j}\le \vep$, $j=0,1$.  Then, $|e(\ga \circ \xi_0)-e(\ga \circ \xi_1)|\le \vep$. Choose $f_0,f_1\in \iso(E)$ such that $|e(\ga\circ \xi_j)- c(f_j \rest X)|\le \vep$ and such that $\nrm{f_j\rest X_0- \ga\circ \xi_j}\le \vep$ for $j=0,1$. Then,
\begin{align*}
|c(h\circ \ga_0)-c(h\circ \ga_1)|\le & \nrm{h\circ \ga_0-f_0\rest X}+ \nrm{h\circ \ga_1-f_1\rest X}+ 3\vep \le \\
\le & \nrm{f_0\rest X- h \circ g_{\ga_0}\rest X}+ \nrm{f_1\rest X- h \circ g_{\ga_1}\rest X}+ 5\vep  \le \\
\le & (1+\de) ( \nrm{f_0\rest X_0- h \circ g_{\ga_0}\rest X_0}+ \nrm{f_1\rest X_0- h \circ g_{\ga_1}\rest X_0 })+ 5\vep\le \\
 \le & (1+\de) ( \nrm{f_0\rest X_0- \ga \circ \xi_0}+ \nrm{f_1\rest X_0-   \ga \circ \xi_1 }+ 6\vep )+ 5\vep \le 21 \vep.   \qedhere
\end{align*}
Since $D$ is $\vep$-dense, it follows from the previous inequality that $\osc(h  \circ \Emb(X,Y))\le 23\vep$.
\let\qed\relax

\fprue

\subsection{The approximate Ramsey property of $\{\ell_\infty^n\}_n$}\label{oi23349283}
The content of this part is the proof of the approximate Ramsey property of the   family $\{\ell_\infty^n\}_n$, and consequently of the class of all finite-dimensional Banach spaces,  over $\mbb F=\R,\C$. Our proof is based on the  \emph{Dual Ramsey Theorem }(DRT) of R. L.  Graham and B. L. 
Rothschild \cite{graham_ramseys_1971}. For convenience, we present its
formulation in terms of rigid surjections between finite linear orderings.
Given two linear orderings $(R,<_{R})$ and $(S,<_{S})$, a surjective map $%
f:R\rightarrow S$ is called a \emph{rigid surjection} when $\min_R
f^{-1}(s_{0})<\min_R f^{-1}(s_{1})$ for every $s_{0},s_{1}\in S$ such that $%
s_{0}<_{S}s_{1}$. Let $\mathrm{Epi}(R,S)$ be the collection of rigid
surjections from $R$ to $S$.
 
\begin{theorem}[(DRT)  \cite{graham_ramseys_1971}]
\label{Theorem:DLT-rigid}For every finite linear orderings $R$ and $S$ such
that $|R|<|S|$ and every $r\in \mathbb{N}$ there exists an integer $n>|S|$
such that, considering $n $ naturally ordered, every $r$-coloring of $%
\mathrm{Epi}(n,R)$ has a monochromatic set of the form $\mathrm{Epi}%
(S,R)\circ \gamma =\{{\sigma \circ \gamma }\,:\,{\sigma \in \mathrm{Epi}(S,R)%
}\}$ for some $\gamma \in \mathrm{Epi}(n,S)$.
\end{theorem}
 
We prove the following. 
 \begin{theorem}\label{io3iorio32io33gtfg} The class $\{\ell_\infty^n\}_{n\in \N}$ satisfies the (SRP) with modulus $\varpi(\de)=\de$.  
 
\end{theorem}
It follows from the KPT correspondence in Theorem \ref{liwjr3iwejirwe} and Proposition \ref{SRP=cSRP} the announced result and Corollary  \ref{knjreiureiu87548954906}. 

\begin{reptheorem}{zcxvbnjhgfd}
 The group of isometries of the Gurarij space endowed with the strong operator topology is extremely amenable. \qed
\end{reptheorem}



\cor\label{knjreiureiu87548954906}
The class of finite-dimensional Banach spaces satisfies the  (SRP) with modulus $\varpi(\de)=\de$.  
\fcor	
We will give a direct proof of the (ARP) of the class of all finite-dimensional Banach spaces later. Coming back to Theorem \ref{io3iorio32io33gtfg}, by means of Proposition 	\ref{SRP=cSRP} we need to prove that $\{\ell_\infty^n\}$ satisfies the stable amalgamation property with modulus $\de$, and that it has the (ARP).     Observe that  a linear map $\ga:\ell_\infty^d\to \ell_\infty^n$ is a $\de$-isometric embedding if and only if its dual operator $\ga^*: \ell_1^n \to \ell_1^d$
satisfies that $\ga^*(\ball(\ell_1^n))\con \ball(\ell_1^d)\con\ga^*( (1+\de)\ball(\ell_1^n))$.  When $\de=1$ such an operator $\sig: \ell_1^n\to \ell_1^d$ satisfying that  $\sig(\ball(\ell_1^n))=\ball(\ell_1^d)$ is called a {\em quotient map}.  A simple argument using extreme points shows that this is equivalent to saying that $\{u_j\}_{j<d}\con S^1(\mbb F) \cdot \{\sig(u_j)\}_{j<n}$, where $S^1(\mbb F)= \conj{a\in \mbb F}{|a|=1}$, and where $u_j$ is the $j^\mr{th}$ unit vector whose only non-zero coordinate has value 1 and it is on the $j^{\mr{th}}$ position.   Let   $\mr{Quo}(\ell_1^n,\ell_1^d)$ be the metric space of quotients. Finally, observe that the dual functor $\ga\in \Emb(\ell_\infty^d,\ell_\infty^n)\mapsto \mr{Quo}(\ell_1^n,\ell_1^d)$ is an isometric bijection.  This means that the (ARP) of $\{\ell_\infty^n\}_n$ is equivalent to the assertion of the following lemma.
\begin{lemma}\label{iojriowjrwe8855}
For every $d,m\in \N$ and $\vep>0$ there is some $n\in \N$ such that every continuous coloring of $\mr{Quo}(\ell_1^n,\ell_1^d)$ $\vep$-stabilizes on $\mr{Quo}(\ell_1^m,\ell_1^d) \circ \ro$ for some $\ro\in \Quo(\ell_1^n,\ell_1^m)$.
\end{lemma}		
Lemma \ref{iojriowjrwe8855} will be proved later using the Dual Ramsey Theorem.  		

%
 
\begin{proposition}\label{li23jior3ji3r}  
$\{\ell_\infty^n\}_n$ is a stable Fraïssé class with modulus $\de$.
\end{proposition}  
\begin{proof}
Suppose that $\ga:\ell_\infty^d\to \ell_\infty^m$, $\eta:\ell_\infty^d\to \ell_\infty^n$ are  $\de$-isometric embeddings. This means that the dual operators $\ga^*:\ell_1^m\to \ell_1^d$ and $\eta^*:\ell_1^n\to \ell_1^d$ satisfy that $\ga^*(\ball(\ell_1^m))\con \ball(\ell_1^d)\con\ga^*( (1+\de)\ball(\ell_1^m))$, and $\eta^*(\ball(\ell_1^n))\con \ball(\ell_1^d)\con\eta^*( (1+\de)\ball(\ell_1^n))$.  We define $\sig:\ell_1^{m+n}\to \ell_1^m$ and $\tau: \ell_1^{m+n}\to \ell_1^n$ as follows. For each $j<m$, choose $y_j\in \ell_1^n$ with $1\le \nrm{y_j}\le 1+\de$  such that $\eta^*(y_j)=\ga^*(u_j)$, and for $k<n$ choose $x_k\in \ell_1^m$ with $1\le \nrm{x_k}\le 1+\de$  such that $\ga^*(x_k)=\eta^*(u_k)$   Now for each $j<m$, let $\sig(u_j):=u_j$ and $\tau(u_j):= y_j/\nrm{y_j}$, and for $k<n$, let $\sig(m+k)= x_k/\nrm{x_k}$ and $\tau(u_{m+k})=u_k$. Then clearly $\sig(\ball(\ell_1^{m+n}))=\ball(\ell_1^m)$ and $\tau(\ball(\ell_1^{m+1}))=\ball(\ell_1^n)$  and $\nrm{ \ga^*\circ \sig-\eta^*\circ \tau}_{\ell_1^{m+n},\ell_1^d}\le \de$. 
\end{proof}

 Our proof of the (ARP) of $\{\ell_\infty^n\}_n$ uses crucially the Dual Ramsey theorem.  
The case $d=1$  was first proved by Gowers \cite{gowers_lipschitz_1992}, indirectly, as it follows easily via a
compactness argument from the oscillation stability of the space $c_0$.   We start by presenting a simple proof of this  result for \emph{positive} embeddings in the real case.
Given integers $k$ and $n$, let $\fin_k(n)$ be the collection of
all mappings from $n$ into $k+1=\{0,1,\dots,k-1,k\}$ such that $k$ is in its range. Let $T:\fin_k(n)\to
\fin_{k-1}(n)$ be the \emph{tetris} operation defined pointwise for $f\in \fin_k(n)$  by $T(f)(i):=
\max\{f(i)-1, 0\}$. Given disjointly supported $f_0,\dots, f_{l-1}$ in $\fin_j(n)$, the combinatorial space
$\langle f_i \rangle_{i<l}$ is the collection of all combinations $\sum_{i<l} T^{k- j_i}(f_i)$ where
$(j_i)_{i<l}\in \fin_k(l)$.

\prop[Gowers]\label{lkjdfkjsdkljfdee373476}
For every $k$, $m$ and every $r$ there is some $n$ such that  every $r$-coloring of $\fin_k(n)$ has a
monochromatic set of the form $\langle f_i\rangle_{i<m}$ for some disjointly supported sequence $(f_i)_{i<m }$
in $\fin_k(n)$.\fprop

In the next, let  $\mathbf{GR}(d,m,r)$ be the minimal $n$ so that  (DRT) holds for the parameters $d$, $m$ and $r$.
\prue[Proof of Proposition \ref{lkjdfkjsdkljfdee373476}]
 Fix $k$, $m$ and $r$. We claim that $n=\mathbf{GR}(k+1, km+1, r)$ works. Fix an $r$-coloring $c$ of  $\fin_k(n)$. We consider $k+1$, $mk+1$, and $n$ canonically ordered. For a subset $A$ of $n$, we let $\mathbbm 1_A$ be the indicator function of $A$. Let  $\Phi: \mr{Epi}(n, k+1)\to \fin_k(n)$ be defined by $\Phi(\sig):=\sum_{i\le k} i \cdot \mathbbm 1_{\sig^{-1}(i)}$. By the Ramsey property of $n$ there is some rigid surjection $\ro:n\to mk +1$ such that $c\circ \Phi$ is constant on $\mr{Epi}(mk +1, k+1) \circ \ro$ with value $\widehat r$.  For each $j<m$, let  $f_j:=\sum_{1\le i \le k}  i\cdot \mathbbm 1_{\ro^{-1}(j k +i)}$. Then  $c$ is constant on $\langle f_j\rangle_{j<m}$. To see this, given $f=\sum_{l<m} T^{k-j_l}f_l\in \langle f_j\rangle_{j<m}$ we define $\sig:m k +1 \to k+1$ by
 $\sig(0)=0$ and $\sig( l k +i):= \max\{ i-k + j_l,0 \}$ for $l<m$ and $1\le i \le k$.  Then  for $0<i_0$ one has that $\min \sig^{-1}(i_0)= k l_0 +(i_0 +k -j_{l_0})$ where $l_0=\min \conj{l<m}{ i_0\le j_l }$, so $\sig$ is a rigid surjection. It is not difficult to see that $\Phi(\sig\circ \ro)=f$, so $c(f)=\widehat{r}$.
\fprue

\prue[Proof of Lemma \ref{iojriowjrwe8855}] 
We start by the following simple fact. 
\clam\label{poikdlsfiop}
There is a finite $\vep$-dense subset $\mc D$ of $\mr{Ball}(\ell_1^d)$ containing $\{u_j\}_{j<d}$ such that for every non-zero $x\in \mr{Ball}(\ell_1^d)$ there is $y\in \mc D$ such that $\nrm{y-x}_1\le \vep$ and $\nrm{y}_1<\nrm{x}_1$.
\fclam	
\prucl
Let $D$ be a   finite $\vep/2$-dense subset   of the unit sphere of $\ell_1^d$ containing $\{u_j\}_{j<d}$, and let  $0=\la_0<\cdots <\la_{p-1}=1$  be such that  $\max_{0\le j\le p-2}\la_{j+1}-\la_j\le \vep/2$.   Then $ \mc D=\bigcup_{k<p} \la_k \cdot D$  satisfies what we want. 
\fprucl
Fix such a $\vep$-dense set $\mc D$, and  let  $\pe$ be any linear ordering of $\mathcal D$ such that if
$\nrm{x}_{1}<\nrm{y}_{1}$ then $x\pe y$.  Let $\mr{emb}(d,m)$ be the collection of all 1-1 mappings $f: d\to m$, and let $S$ be a finite $\vep$-dense subset of $S^1(\mbb F)$.   For each $(f,\theta)\in \mr{emb}(d,m)\times S^d$, let $h_{f,\theta}:\ell_1^d\to \ell_1^m$ be the linear map obtained by setting $h_{f,\theta}(u_j):=\theta_j \cdot u_{f(j)}$. Then clearly  $h_{f,\theta}$ is an isometric embedding from $\ell_1^d$ into $\ell_1^m$.
\clam\label{oi3rioiowrioweijoer}
For every $T\in \Quo(\ell_1^m,\ell_1^d)$ there is a pair $(f,\theta)\in \mr{emb}(d,m)\times S^d $ such that $\nrm{T\circ h_{f,\theta}-\id_{\ell_1^d}}_{\ell_1^d, \ell_1^d}\le \vep$.
\fclam
\prucl
For each $k<d$ choose $f(k)< m$ such that $T(u_{f(k)})=  a_k u_k$ where $|a_k|=1$. Clearly $k<d\mapsto f(k)$ is an injection from $d$ into $m$. For each $k<d$, let  $\theta_k\in S$ be such that $|1/a_k-\theta_k|\le \vep$, and let $\theta:=(\theta_k)_k$. Then 
\begin{align*}
  \nrm{T  \circ h_{f,\theta}-\id_{\ell_1^d}}_{\ell_1^d, \ell_1^d}  = & \max_{k<d} \nrm{T  \circ h_{f,\theta}(u_k)-u_k}_1=\max_{k<d} \nrm{T (\theta_k u_{f(k)})-u_k}_1=\\ 
  =& \max_{k<d} \nrm{a_k\theta_k u_{k}-u_k}_1 \le  \max_{k<d} |a_k\theta_k -1| \le \vep. \qedhere
  \end{align*}
\fprucl
Let $\De:= \mathcal D \times \mr{emb}(d,m) \times S^d$ be
ordered by the lexicographical ordering induced from $\mathcal D$ ordered by $\pe$, and $\mr{emb}(d,m)\times S^d$  ordered arbitrarily.  We claim that $n:=  \mathbf{GR}(|\mc D|, |\De|,r)$ works. Indeed, let
$c$ be an $r$-coloring of $\mr{Quo}(\ell_1^n,\ell_1^d)$. We define an injection  $\Phi: \mr{Epi}( n, \mathcal D)\to \Quo(\ell_1^n,\ell_1^d)$ by assigning to each $\sig\in \mr{Epi}(n,\mathcal D)$ the  operator $T:=\Phi(\sig):\ell_1^n\to \ell_1^d$  such that for each $\xi<n$ one has that $T (u_\xi):=\sig(\xi)$. Equivalently the $\xi^\mr{th}$-column vector  of the matrix corresponding to $\Phi(\sig)$ in the respective unit bases is $\sig(\xi)$.  It is easily verified that $T$ is always a quotient map.    It follows by the Dual Ramsey Theorem   applied to the coloring $\widehat{c}:= c\circ \Phi$
that there is $\ga_0\in \mr{Epi}(n, \De)$ such that
\begin{equation*}
\widehat{c} \text{ is constant on $\mr{Epi}(\De,\mathcal D) \circ    \ga_0 $ with value $r_0<r$.}
\end{equation*}
Let $R\in \Quo(\ell_1^n,\ell_1^m)$ be the quotient such that, for every $\xi<n$, one has that $R(u_\xi)=h_{f,\theta}(v)$, where $(v,f,\theta)=\ga_0(\xi)$.
%
%
The proof is finished once we establish the
following.
\clam
For every $T\in \Quo(\ell_1^m,\ell_1^d)$ there exists $\phi\in \mr{Epi}(\De,\mathcal D)$ such that
\begin{equation*}
\nrm{\Phi(\phi \circ \ga_0) - T\circ R}_{\ell_1^n,\ell_1^d}\le \vep.
\end{equation*}
\fclam
\prucl
Fix $T\in \Quo(\ell_1^m,\ell_1^d)$, and    use   Claim \ref{oi3rioiowrioweijoer} to  choose $(\bar f,\bar \theta)\in \mr{emb}(d,m)\times S$ such that $\nrm{T\circ h_{\bar f,\bar \theta}- \id_{\ell_1^d}}_{\ell_1^d,\ell_1^d}\le \vep$.  Now we define   $\phi: \De \to \mc D$
as follows. Fix $(v,f,\theta)\in \De$. 
\begin{enumerate}[(i)]
\item  If $T (h_{f,\theta}(v))=0$, then we set $\phi(v,f,\theta):=0$.
\item Suppose that  $T(  h_{f,\theta}(v))\neq 0$; if $(f,\theta)=(\bar f,\bar \theta)$, then we set $\phi(v,f,\theta):=v$; otherwise,  we set $\phi(v,f,\theta):=w$ where  $w\in \mc D$ is such that  $\nrm{T(h_{f,\theta}(v))-w}_{\ell_1}\le \vep$, and such that $\nrm{w}_1<\nrm{T(h_{f,\theta}(v))}_1$.

\end{enumerate}  
We see   that $\phi:\De\to D$ is a rigid surjection. First, $\min \phi^{-1}(0)=(0,f_0,\theta_0)$, where $(f_0,\theta_0)$ is the minimum of $\emb(d,m)\times S^d$. Now suppose that  $v\in \mc D$ is a non zero vector.    We prove that 
$\min \phi^{-1}(v)=(v,\bar f,\bar \theta)$: Suppose that $\phi(u,f,\theta)=v$, and $(f,\theta)\neq (\bar f,\bar \theta)$. By the definition of $\phi$,  $\nrm{v}_1<\nrm{T(h_{f,\theta}(u))}_1\le \nrm{T}_{\ell_1^m, \ell_1^d}\cdot \nrm{h_{f,\theta}(u)}_1\le \nrm{u}_1$, because $T$ is a contraction and $h_{f,\theta}$ is an isometric embedding.  Hence, $v\pe u$, and since in $\De$ we are considering the lexicographic ordering, $(v,\bar f,\bar \theta)<(u, f,\theta)$. Since obviously $\phi(v,\bar f,\bar \theta)=v$, we obtain that $\min \phi^{-1}(v)=(v,\bar f,\bar \theta)$. Hence, if $0\neq v < w$, then $\min \phi^{-1}(v)<\min \phi^{-1}(w)$. 

Finally, we estimate  $  \nrm{\Phi(\phi\circ \ga_0)- T\circ R}_{\ell_1^n,\ell_1^d}= \max_{\xi<n} \nrm{\Phi(\phi\circ \ga_0)(u_\xi)- T(R(u_\xi))}_{\ell_1^d}$.   Fix $\xi<n$, and suppose that $\ga_0(\xi)=(v,f,\theta)$.  Then by definition, $(\Phi(\phi\circ \ga_0))(u_\xi)= \phi(\ga_0(\xi))$, and $T(R(u_\xi))= T(h_{f,\theta}(v))$.  Now 
\begin{enumerate}[(a)]
\item if $T(h_{f,\theta}(v))=0$, then $0=\phi(v,f,\theta)=\Phi(\phi\circ \ga_0)(u_\xi)$ and $0= T(h_{f,\theta}(v))=T(R(u_\xi))$.
\item If $T(h_{f,\theta}(v))\neq 0$ and  $(f,\theta)=(\bar f,\bar \theta)$, then $ \Phi(\phi\circ \ga_0)(u_\xi)=\phi(v,f,\theta)=v$ while  $T(R(u_\xi))= T(h_{f,\theta}(v))=w$  is such that $\nrm{w-v}_1\le \vep$.
\item If $T( h_{f,\theta}(v))\neq 0$ and  $(f,\theta)\neq (\bar f,\bar \theta)$, then $ \Phi(\phi\circ \ga_0)(u_\xi)=\phi(v,f,\theta)=w$ is chosen such that $\vep\ge \nrm{w- T(h_{f,\theta}(v))}_1=\nrm{w-T(R(u_\xi))}_1$.\qedhere
\end{enumerate}  
%
%

\fprucl

\let\qed\relax
\fprue
%
%
%
%
%

\subsection{  (ARP) of Polyhedral spaces and finite-dimensional spaces}\label{uiuhyuuy893232321}
We give   an explicit proof of the approximate Ramsey property of the class of finite-dimensional polyhedral spaces. This is done by using injective envelopes of polyhedral spaces, and then  by reducing colorings of polyhedral spaces to colorings of $\ell_\infty^n$-spaces. We also use this to explicitly prove the (ARP) of the class of all finite-dimensional Banach spaces.
 In this way, knowing the number of extreme points of the dual unit ball of given spaces, one can estimate upper bounds of the  corresponding   Ramsey numbers. For simplicity, we present the proof in the case of real Banach spaces. Thus, all the Banach spaces are assumed to be real in this section.
\defi
A  finite-dimensional  space $F$  is called \emph{polyhedral} when its unit ball $\mr{Ball}(F)$ is a
polyhedron, i.e., when the set $\partial_e(\ball(F))$ of \emph{extreme} points of   $\ball(F)$  is finite.
\fdefi
The spaces $\ell_\infty^n$ and $\ell_1^n$ are polyhedral. In fact,  a   finite-dimensional  space  is
polyhedral if and only if its dual ball is  polyhedral. It follows from this, a separation argument, and the
Milman theorem,  that a  finite-dimensional  space $F$ is polyhedral if and only if there is a finite set
$A\con \mr{Sph}(F^*)$ such that $\nrm{x}=\max_{f\in A} f(x)$ for every $x\in F$.  Also, every subspace of a
polyhedral space is   polyhedral, and every finite-dimensional polyhedral space embeds into $\ell_\infty^n$ for some $n\in \mathbb{N}$.

\defi[Polyhedral spaces]
Given an integer $d$, let $\mr{Pol}_{d}$  be the class of all polyhedral spaces  $F$  such that
$\#\partial_e(B_{F^*})=2d$. Given $d,m\in \N$,  $r\in \N$   and $\vep>0$,  let $\mathbf{n_{pol}}(d,m,r,\vep)$
be the minimal integer $n\ge m$ such that for every $F\in \mr{Pol}_{d}$ and $G\in \mr{Pol}_{m}$,  every
$r$-coloring of $ \mr{Emb}(F,{\ell_\infty^n})$  has an $\vep$-monochromatic set of the form $T\circ
\mr{Emb}(F,G) $ for some $T\in \mr{Emb}(G,{\ell_\infty^n})$.
\fdefi

\defi[Injective envelope of a polyhedral space]
The injective envelope of a polyhedral space $F$ is a pair $(n_F,\Psi_F)$, where $n_F$ is an integer and   $\Psi_F\in \mr{Emb}(F,\ell_\infty^{n_F})$ such that for every   isometric embedding $T:F\to \ell_\infty^n$ there is an isometric embedding $U:\ell_\infty^{n_F}\to \ell_\infty^n$ such that  $T= U\circ \Psi_F$.
\fdefi

\prop\label{8998786444oig1}
$\mathbf{n_{pol}}(d,m,r,\vep)=\mathbf{n_{\infty}}(d,m,r,\vep)$.
\fprop

\prue[{Proof of Proposition \ref{8998786444oig1}}]
First of all, $\ell_\infty^k\in \mr{Pol}_k$, so $\mathbf{n_{pol}}(d,m,r,\vep)\ge
\mathbf{n_{\infty}}(d,m,r,\vep)$. Fix now  an $r$ coloring $c$ of   $\mr{Emb}(F,{\ell_\infty^n})$. Let
$\widehat{c}:\mr{Emb}({\ell_\infty^{d} },{\ell_\infty^n}) \to r$
 be defined for  $U\in \mr{Emb}(\ell_\infty^{d},\ell_\infty^n)$ by
$ \widehat{c}(U):= c (U\circ \Psi_F)$. Let $\widehat{T}\in \mr{Emb}({\ell_\infty^{m}} ,{\ell_\infty^n})$
and $\widehat{r}<r$ be such that
\begin{equation}\label{klnfldshf732}
\widehat{T}\circ \mr{Emb}({\ell_\infty^{d}}, {\ell_\infty^{m}})\con (\widehat{c}^{-1}\{\widehat{r}\})_\vep.
\end{equation}
Let $T:=\widehat{T}\circ \Psi_G.$ We claim that $T\circ \mr{Emb}(F,G) \con (c^{-1}\{\widehat{r}\})_\vep$.
Let $U\in  \mr{Emb}(F,G)$, and let $W\in \mr{Emb}(\ell_\infty^{d}, \ell_\infty^{m})$ be such that $\Psi_G\circ U=W\circ \Psi_F$. From the
inclusion in \eqref{klnfldshf732} there exists   $V\in \mr{Emb}(\ell_\infty^{d},\ell_\infty^{m})$ such that
$\widehat{c}(V)=\widehat{r}$ and $\nrm{V-\widehat{T}\circ W}< \vep$. Let $\widehat{V}:=V\circ
\Psi_F$. Then $c(V\circ \Psi_F)=\widehat{c}(V)=\widehat{r}$, while
\[  
\nrm{\widehat{V}-T\circ U}=   \nrm{V\circ \Psi_F-\widehat{T} \circ \Psi_G\circ U}=\nrm{V\circ \Psi_F-\widehat{T} \circ W\circ \Psi_F}\le  \nrm{V-\widehat{T} \circ W}< \vep.\qedhere
 \] \let\qed\relax
\fprue

\subsubsection{Approximate Ramsey property for  finite-dimensional  normed spaces}\label{uiuhyuuy893232322}
We give an explicit, constructive proof of  approximate Ramsey property  arbitrary  finite-dimensional   normed spaces.   The proof is based on the approximate Ramsey property of polyhedral spaces and  the well known fact that the finite-dimensional polyhedral spaces are dense in the class of  finite-dimensional  normed spaces with respect to the
Banach-Mazur distance.  In fact, we have the following.
\prop\label{approx_by_polyh}
Suppose that $\dim X=k$. For every $0<\vep<1$ there is  a polyhedral space $X_0\in \mr{Pol}_d$ such that $d_\mr{BM}(X,X_0)\le \vep$, where $d\le ((2+3\vep)/\vep)^k$.
\fprop
\prue
Let $\de:=\vep(1+\vep)^{-1}$. Let $D\con \mr{Sph}(X^*)$ be a finite $\de$-dense subset of $S_{X^*}$ of cardinality $\le (1+2\de^{-1})^k=((2+3\vep)/\vep)^k$
(see for example \cite[Lemma 2.6]{milman_asymptotic_1986}). On $X$ we define the polyhedral norm $N(x):=\max_{f\in D} |f(x)|$. It follows that $X_0:=(X,N)\in \mr{Pol}_{d}$ with $d\le \#D$, and $d_{\mr{BM}}(X,X_0)\le \vep$.
\fprue

\defi
  Given $X$ of finite dimension, and $\theta\ge 1$, let
$\Eemb_\theta(X,Y)$ be the collection of all 1-1 mappings $T:X\to Y$ such that $1\le \nrm{T},\nrm{T^{-1}}$ and $\nrm{T}\cdot \nrm{T^{-1}}\le \theta$.

Let $(X_i)_{i\le n}$ be a sequence of Banach spaces. We say that a  pair $(Y,J)$
of a Banach space $Y$ and $J\in \Emb(X_{n},Y)$ is   \emph{$(\theta,\tau)$-correcting} for
$\bar X$ ($1<\theta<\tau$) when every $X_i$ isometrically embeds into $Y$, and
for every $j<n$ and every $\ga\in \Eemb_\theta(X_j,X_n)$ there exists $I_\ga\in \mr{Emb}(X_j,Y)$ such that $\nrm{J \circ \ga- I_\ga}<\tau-1$.
\fdefi

\prop\label{oi43jioj4ffr333}
Every finite sequence of  finite-dimensional  spaces $(X_i)_{i\le n}$ and every $1<\theta<\tau$ has  a
$(\theta,\tau)$-correcting pair $(Y,J)$.  
Moreover, when each $X_j$ is polyhedral,  then $Y$ can be taken polyhedral. 

\fprop
\prue
The proof is by induction on $n\ge 1$.  Suppose first that $n=1$.
A simple inductive argument, where the case $\#\mc N=1$ is proved by  Kubis and  Solecki in \cite[Lemma 2.1]{kubis_proof_2013}, gives the following.
%
\clam\label{clma_approx}
Suppose that $\mathcal N\con \Eemb_\theta(X_0,X_1)$ is finite. Then there exist a
 finite-dimensional  space $Y$ 
 and $\Theta\in \mr{Emb}(X_1,Y)$ such that for every  $T\in \mathcal N$ there
is $I\in \mr{Emb}(X_0,Y)$  such that $\nrm{I-\Theta\circ T}< \theta-1$.
\fclam
Let   $\mathcal N$ be a finite $(\tau-\theta)$-net of
$\Eemb_\theta(X_0,X_1)$.
%
  Then the pair $(X,I)$ obtained by applying Claim \ref{clma_approx} to $\mc N$ is $(\theta,\tau)$-correcting for $(X_0,X_1)$.
Now suppose that $n>1$. Find a $(\theta,\tau)$-correcting pair $(Y_0, \Theta_0)$  for $(X_j)_{j=1}^n$
 Let  $\mathcal N$ be a finite $(\tau-\theta)$-net of
$\Eemb_\theta(X_0,X_n)$.  
Let
$(Y,\Theta_1)$  be a pair obtained by applying Claim \ref{clma_approx} to $\Theta_0\circ \mc N$. It can be easily verified that $(Y,\Theta_1\circ
\Theta_0 )$ is a $(\theta,\tau)$-correcting pair for $(X_j)_{j\le n}$. 
 \fprue

\begin{theorem}
 The class of all finite-dimensional Banach spaces $\mr{FdBa}$ has the (SRP). 
\end{theorem}

\prue 
We know that $\mr{FdBa}$ is a stable Fraïssé class, so we only have the proof that it satisfies the discrete (ARP). 
Fix finite-dimensional spaces $F$, $G$, $r\in \N$, $\vep>0$,  and set $\de:= \vep/5$.   Let
$F_0\in \mr{Pol}_d$, $G_0$ be polyhedral, and  surjective isomorphisms $\Phi_F:F\to F_0$ and $\Phi_G: G\to
G_0$  such that $\nrm{\Phi_F}=\nrm{\Phi_G}=1$ and $\nrm{\Phi_F^{-1}},\nrm{\Phi^{-1}_G}<1+\vep/5$.  Notice that $d$ can be taken such that $d\le  ( (10+3\vep)/\vep    )^{\dim F}$. Let
\begin{enumerate}[(i)]
\item  $(H_0,\Theta_0)$ be a  $(1+\vep/5,1+\vep/4)$-correcting pair for $(F_0,G_0)$ with $H_0\in \mr{Pol}_m$, and let 
\item $(H,\Theta_1)$ be a
$(1+\vep/5,1+\vep/4)$-correcting pair for the triple $(F,G,\ell_\infty^n)$ 
where $n:=\mbf{n_\mr{pol}}(d,m,r,\vep/4)$. 
  
\end{enumerate}  
 We claim that $H$ works. Fix $c:\mr{Emb}(F,H)\to r$. Let
$\widehat{c}:\mr{Emb}(F_0,\ell_\infty^n)\to r$ be the induced coloring defined for $\ga\in
\mr{Emb}(F_0,\ell_\infty^n)$ by choosing $I_\ga\in \mr{Emb}(F,H)$ such that
$\nrm{I_\ga - \Theta_1\circ \ga \circ \Phi_F}<\vep/4$  and declaring $\widehat{c}(\ga):=c(I_\ga)$. By the Ramsey property of $n$, there exists $\ro\in \mr{Emb}(H_0,\ell_\infty^n)$ and $\widehat r<r$ such that $\ro \circ \mr{Emb}(F_0,H_0)\con (\widehat{c}^{-1}\{\widehat r\})_{\vep/4}$. Let $S\in \mr{Emb}(G,H)$ be such that
\begin{equation}\label{oi43irro34i4556gghn}
\nrm{S-\Theta_1\circ \ro \circ \Theta_0\circ \Phi_G}<\frac\vep4.
\end{equation}
\clam
$S\circ \mr{Emb}(F,G)\con (c^{-1}(\widehat r))_{\vep}$.
\fclam
\prucl
Fix $T\in \mr{Emb}(F,G)$. Let $\tau\in \mr{Emb}(F_0,H_0)$ be such that $\nrm{\tau- \Theta_0\circ \Phi_G\circ T\circ \Phi_F^{-1}}<\vep/4$. This is possible because $\Phi_G\circ T\circ \Phi_F^{-1}\in \Eemb_{1+\vep/5}(F_0,G_0)$. Let now $\ga\in
\mr{Emb}(F_0,\ell_\infty^n)$ be such that $\widehat{c}(\ga)=\widehat{r}$ and
$ \nrm{\ga- \ro \circ \tau}<\vep/4$.
Then, $c(I_\ga)=\widehat r$ and  $\nrm{ S\circ T-I_\ga }<\vep$.  It follows from \eqref{oi43irro34i4556gghn}
and the fact that the operator $T$ satisfies that $\nrm{T}=1$, that
\begin{equation*}\label{uuuuunnnno1}
\nrm{S\circ T - \Theta_1 \circ \ro \circ \Theta_0\circ \Phi_G\circ T}<\frac\vep4.
\end{equation*}
This is the diagram:

 
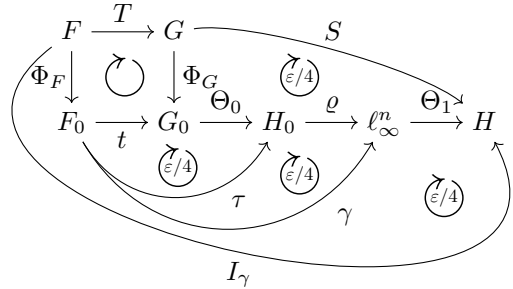
\begin{figure}[H]
\begin{tikzpicture}[descr/.style={fill=white,inner sep=2pt}]

\matrix (m) [matrix of math nodes, row sep=2em, column sep=2em, text height=1.5ex, text depth=0.25ex]
{     F   &   G   & & &   \\
 F_0 &  G_0 & H_0 & \ell_\infty^n & H
\\
 };

\path[->,font=\normalsize]

(m-1-1) edge node[above] {$T$} (m-1-2)
             edge [bend right=130,looseness=1.6] node[below]  {$I_\ga$} (m-2-5)
(m-1-1) edge node[right]   {}  (m-2-1)

(m-1-2) edge node[right] {}   (m-2-2)
             edge [bend left=25,looseness=.5] node[above]  {$S$} (m-2-5)
(m-2-1) edge node[below]  {$t$} (m-2-2)
             edge [bend right=60,looseness=1.2]  node[above]  {} (m-2-3)
             edge  [bend right=60,looseness=1.2]  node[right]  {} (m-2-4)

(m-2-2) edge node[above]  {$\Theta_0$} (m-2-3)

(m-2-3) edge node[above]  {$\ro$} (m-2-4)

(m-2-4) edge node[above]  {$\Theta_1$} (m-2-5)


;

\node at (-2,  0) (node1) {\scalebox{2.2}{$\circlearrowright$} };
\node at (0.3,  -0) (S1) {\scalebox{2.2}{$\circlearrowright$} }; \node at (0.3,  -0) (S2)
{\scalebox{0.7}{${\vep/4}$} };

\node at (-1.3,  -1.2) (tau1) {\scalebox{2.2}{$\circlearrowright$} }; \node at (-1.3,  -1.2) (tau2)
{\scalebox{0.7}{$\vep/4$}  };
\node at (0.3,  -1.3) (node6) {\scalebox{2.2}{$\circlearrowright$} }; \node at (0.3,  -1.3) (node7)
{\scalebox{0.7}{$\vep/4$} };
\node at (2.2,  -1.6) (node8) {\scalebox{2.2}{$\circlearrowright$} }; \node at (2.2,  -1.6) (node9)
{\scalebox{0.7}{$\vep/4$} };
\node at (-3,  0) (F) {\scalebox{1}{$\Phi_F$} };
\node at (-1.0,  0) (G) {\scalebox{1}{$\Phi_G$} };

\node at (0.9,  -1.85) (ga) {\scalebox{1}{$\ga$} };

\node at (-.50,  -1.65) (tau) {\scalebox{1}{$\tau$} };

\end{tikzpicture}
\end{figure}
\noindent Consequently, $\nrm{S\circ T -I_\ga}<\vep$.
 \fprucl
 \let\qed\relax
\fprue

\subsection{Finite metric spaces}\label{uiuhyuuy893232326}
 Recall that the Urysohn space $\mbb U$ is the unique (up to isometry) ultrahomogeneous universal separable complete metric space. Pestov proved in \cite{pestov_ramsey-milman_2002} that the group  $\iso(\mbb U)$ of surjective isometries of $\mbb U$ is extremely amenable, using the method of concentration of measure. It is also proved   a version of the (KPT) correspondence for $\iso(\mbb U)$, that gives as a consequence the following  the   (ARP) of finite metric spaces. 
 

\teor\label{lkwklerkewrwerrr}
For every finite metric spaces $M$ and $N$, $r\in \N$ and $\vep>0$ there exists  a finite metric space $P$  such that every $r$-coloring  $\mr{emb}(M,P)$ has a $\vep$-monochromatic set of the form  $\sig \circ \mr{emb}(M,N)$ for some $\sig\in \mr{emb}(N, P)$.
\fteor
In the previous statement  $\mr{emb}(M,P)$ is the collection  of all isometric embeddings from   $(M,d_M)$ into  $(N,d_N)$, endowed with the uniform metric  $d(\sig,\tau) := \max_{x\in M} d_N(\sig(x),\tau(x))$. 
Later,  Ne\v{s}et\v{r}il established the (exact) Ramsey property of finite ordered metric spaces   \cite{nesetril_metric_2007}, that is, for every finite \emph{ordered} metric spaces $X$ and $Y$ and every $r\in \N$  there exists a   finite ordered  metric space $Z$ such that for every $r$-coloring of the set $\binom{Z}{X}_<$ of order isometric copies of $X$ in $Z$ there exists  an order isometric copy $Y_0$ of $Y$ in $Z$ such that $\binom{Y_0}{X}_<$ is monochromatic. This  gives another proof of the extreme amenability of $\iso(\mbb U)$. We present here a third proof, which uses the approximate Ramsey property of the class of  finite-dimensional  polyhedral spaces.

Recall that a \emph{pointed} metric space $(X,d,p)$ is a  metric space $(X,d)$ with a distinguished point $p\in X$. Given two pointed metric spaces $(M,p)$ and $(N,q)$, let $\mr{emb}_0(M,N)$ be the set of pointed isometric embeddings, that is, all isometric embeddings from $M$ into $N$ sending $p$ to $q$.
Recall that when  $X$ and $Y$ are normed spaces, we use  $\mr{Emb}(X,Y)$ to denote \emph{linear} isometric embeddings.

%
%

\defi
Given a pointed metric space $(M,d,p)$,  let $\mr{Lip}_0(M,p)$ be the Banach space of all Lipschitz maps $f:M\to \R$ such that $f(p)=0$ endowed with the \emph{Lipschitz} norm,
\begin{equation*}
\nrm{f}:=\sup\left\{\frac{|f(x)-f(y)|}{d(x,y)}\, :\, x\neq y \in X\right\}.
\end{equation*}
Let $\mc F(M,p)$ be the \emph{(Lipschitz) free space} over the pointed metric space $(M,p)$ defined as the closed linear span of the \emph{molecules} $\{\de_x - \de _p\}_{x\in M}$ in the dual space $\mr{Lip}_0(M,p)^*$, where $\de_x$ for $x\in X$ denotes the \emph{evaluation functional} at $x$.  It turns out that $\mc F(M,p)^*$ is isometric to $\mr{Lip}_0(M,p)$.
\fdefi
 It is well-known that $\mr{Lip}_0(M,p)$ does not depend, isometrically, on the choice of the point $p$,  so the corresponding predual will be denoted by $\mc F(M)$.
The space $\mc F(M)$ is also known as the \emph{Arens-Eells} space. More information on Arens--Eells spaces can be found in \cite[Section 2.2]{weaver_lipschitz_1999}.   It is easy to see that the mapping $x\in M\mapsto \de_x\in \mc F(M)$  is an isometric embedding.  Given  finite metric spaces $M$ and $N$ such that $M$ isometrically embeds into $N$,  let $M_\infty=M\cup \{p_\infty\}$,   $N_\infty:=N\cup \{q_{\infty}\}$ be one-point extensions of $M$ and $N$ with the distance $d(p_\infty,x)= d(q_\infty,y):= \min_{z\neq t\in N} d(z,t)$ for every $z\in M$, $y\in N$. Clearly $M_\infty$ and $N_\infty$ are metric spaces.
\prop
Suppose that $M$ and $N$ are metric spaces. Then every isometric embedding $\sig:M\to N$  extends to a unique linear isometric embedding  $T_\sig: \mc F(M_\infty,p_\infty)\to \mc F(N_\infty,q_\infty)$. \qed
\fprop
The proof is a straightforward use of  a standard duality argument, the McShane-Whitney extension  Theorem
for Lipschitz functions \cite[Theorem 1.5.6]{weaver_lipschitz_1999}, and the fact that $\de_{p_\infty}=0$ in $\mc F(M_\infty,p_\infty)$ and
$\de_{q_\infty}=0$ in $\mc F(N_\infty,p_\infty)$.
\prop
If $M$ is a finite metric space, then $\mc F(M)$ is a  finite-dimensional   polyhedral space.
\fprop
\prue
Observe that for each $x\neq y$ in $M$, $\mu_{x,y}:=(\de_x-\de_y)/d(x,y)$ has norm 1 in  $\mr{Lip}_0(M)$  since clearly $\nrm{\mu_{x,y}}\le 1$, and the mapping $d_x(t):=d(x,t)$ for each $t\in M$ is 1-Lipschitz and $\mu_{x,y}(d_x)=1$. It follows from the definition of the Lipschitz norm that the convex hull of $\{\mu_{x,y}\}_{x\neq y \text{ in $M$}}$ is equal to $B_{\mc F(M)}$.
\fprue
\lema\label{iojriojif332223}
Suppose that $M$ and $N$ are two finite metric spaces, suppose $r\in \N$, and that $\vep>0$.  Let $\ro:=\mr{diam}(N)$. Then there exists $n\in \N$ such that  every  $r$-coloring of
$\mr{emb}(M, \ro \cdot B_{\ell_\infty^n} )$ has an $\vep$-monochromatic set of the form  $\sig \circ \mr{emb}(M,N)$ for some $\sig\in \mr{emb}(N, \ro \cdot B_{\ell_\infty^n} )$.
\flema
\prue
Fix finite pointed  metric spaces $(M,p)$, $(N,q)$, $r$ and $\vep>0$. We assume that $M$ isometrically embeds into $N$ since otherwise the statement above is trivially true. Let $d,m$ be such that  $\mc F(M_\infty)\in \mr{Pol}_d$ and $\mc F(N_\infty)\in \mr{Pol}_m$.  Then $n:=\mathbf{n_\mr{pol}}(d,m,r,\vep_0)$, for $\vep_0=\vep/\mr{diam}(M)$ works. Fix a coloring $c:\mr{emb}(M,\ro \cdot B_{\ell_\infty^n})\to r$.  Define
$\widehat{c}:\mr{Emb}(\mc F(M_\infty),\ell_\infty^n)\to r$   by $\widehat{c}(\ga)=c(\sig_\ga)$, where $\sig_\ga: M\to \ro \mr{Ball}(\ell_\infty^n)$ is defined by $\sig_\ga(x)= \ga(\de_x)$ for every $x\in M$. This is well defined since $\nrm{\de_x}=\nrm{\de_x- \de_p}\le d(x,p)\le \mr{diam}(M)\le \mr{diam}(N)$, where the last inequality holds since $\mr{Emb}(M,N)\neq \buit$.  Let $\bar\al\in \mr{Emb}(\mc F(N_\infty), \ell_\infty^n)$ and $\bar r<r$ be such that $\bar \al \circ \mr{Emb}(\mc F(M_\infty),\mc F(N_\infty))\con (\widehat{c}^{-1}(\bar r))_{\vep_0}$.  Let $\bar \tau: N\to \ro \mr{Ball}(\ell_\infty^n)$ be the embedding defined by $\bar \tau (x)= \bar \al(\de_x)$. We claim that $\bar \tau$ works. In fact, $\bar \tau \circ \mr{emb}(M,N)\con (c^{-1}(\bar r))_\vep$.  Let $\sig\in \mr{emb}(M,N)$. Then there exists a unique extension  $\ga_\sig\in \mr{Emb}(\mc F(M_\infty), \mc F(N_\infty))$. Let $\psi\in \mr{Emb}(\mc F(M_\infty), \ell_\infty^n)$ be such that $\widehat{c}(\psi)=\bar r$ and $\nrm{\psi- \bar \al \circ \ga_\sig}<\vep_0$.   Then $\sig_\psi(x)= \psi(\de_x)$ for every $x\in M$ satisfies that $c(\sig_\psi)=\bar r$ and
\[ \pushQED{\qed}
d(\sig_\psi, \bar \tau \circ \sig)=\max_{x\in M}\nrm{\psi(\de_x)- \bar \al( \de_{\sig( x)}) }_\infty=  \nrm{\psi(\de_x)- \bar \al( \ga_\sig(\de_x ) }_\infty <\vep_0 \cdot \mr{diam}(M) =\vep. \qedhere
\popQED
\] \let\qed\relax
\fprue

\prue[{\sc Proof of Theorem \ref{lkwklerkewrwerrr}}]
This is a consequence of Lemma \ref{iojriojif332223}, via a compactness argument.
Fix $M$, $N$, $r$ and $\vep>0$.  Let $n$ be obtained from $M$, $N$, $r$ and $\vep/3$ by applying Lemma \ref{iojriojif332223}. Let $\ro:=\mr{diam}(N)$.   Since $M$ and  $N$ are  finite and $\ro \mr{Ball}(\ell_\infty^n)$ is compact, there exists $P\con \ro \mr{Ball}(\ell_\infty^n)$ finite such that
\begin{equation*}
\mr{emb}(M,\ro \mr{Ball}(\ell_\infty^n)) \con (\mr{emb}(M, P ))_{\frac\vep{3}} \text{ and } \mr{emb}(N,\ro \mr{Ball}(\ell_\infty^n)) \con (\mr{emb}(N, P ))_{\frac\vep{3}}.
\end{equation*}
We claim that $(P,d_\infty)$ works. To this end, let $c:\mr{emb}(P,A)\to r$. Let $\widetilde{c}:\mr{emb}(M, \ro \mr{Ball}(\ell_\infty^n))\to r$ be defined by $\widetilde{c}(\ga)= c(\sig_\ga)$ where $\sig_\ga\in \mr{emb}(M,A)$ is chosen such that $d(\ga,\sig_\ga)<\vep/3$.  By the property of $n$,   there is $\ga\in \mr{emb}(N,\ro \mr{Ball}(\ell_\infty^n))$ and  $\bar r<r$ such that $\ga\circ \mr{emb}(M,N)\con (\widetilde{c}^{-1}(\bar r))_{\vep/3}$. Let $\bar \ga\in \mr{emb}(N,P)$ be such that $d(\ga,\bar \ga)<\vep/3$. It takes a simple computation to see that $\bar \ga \circ \mr{emb}(M,N)\con (c^{-1}(r))_\vep$.
\fprue

\subsection{The closed bifaces of the Lusky simplex and $\mathbf{R}$-Banach spaces}\label{lio3j4oirjir4488}

 There is a natural correspondence between
Banach spaces and  those compact spaces which are {\em absolutely convex}. In the real case, by a compact absolutely
convex set we mean a compact subset of a locally convex topological real
vector space that is closed under absolutely convex combinations of the form 
$\mu x+\lambda y$ for $\lambda ,\mu \in \mathbb{R}$ such that $\left\vert
\lambda \right\vert +\left\vert \mu \right\vert \leq 1$.  Any compact
absolutely convex set $K$ has a canonical involution $\sigma $ mapping $x$
to $-x$. A real-valued continuous function $f$ on $K$ is symmetric if $%
f\circ \sigma =-f$. Similarly, a continuous affine function between compact
absolutely convex sets is symmetric if it commutes with the given
involutions. 
So, given a Banach space $X$,   the unit ball $\mathrm{%
Ball}(X^{\ast })$ of the dual space of $X$ is a compact absolutely convex
set when endowed with the w*-topology. Any compact absolutely convex set $K$
is of this form, where $X$ is the Banach space $A_{\sigma }(K)$ of
real-valued symmetric affine continuous functions on $K$ endowed with the
supremum norm. Each contraction $T:X\to Y$ induces  a  symmetric affine continuous function $T^*:B_{Y^*}\to B_{X^*}$, and vice versa, a given   symmetric affine continuous function $\xi: K\to L$ induces a contraction $\widehat{\xi}:A_\sig(L)\to A_\sig(K)$ by composition.  Furthermore, such a correspondence is functorial, and induces
an equivalence of categories.
The following definition has been introduced in \cite[Section 6.1]{lupini_fraisse_2016}.

\begin{definition} 
A {\em Lazar simplex} is any compact absolutely convex that is   affinely homeomorphic to the unit ball of the dual of a Lindenstrauss space.
\end{definition}					
Lazar simplices have been internally characterized by A. J.  Lazar in \cite{lazar_unit_1972} in terms
of a uniqueness assertion for boundary representing measures, reminiscent of
the analogous characterization of Choquet simplices due to Choquet \cite[%
Section II.3]{alfsen_compact_1971}; see also Subsection \ref%
{Subs:operator_systems} below.  The Lazar simplex corresponding to the
Gurarij space is denoted by $\mathbb{L}$ and called the \emph{Lusky simplex }%
in \cite[Section 6.1]{lupini_fraisse_2016}. It is proved in \cite%
{lupini_fraisse_2016,lusky_gurarij_1976,lusky_construction_1979} that $%
\mathbb{L}$ plays the same role in the category of metrizable Lazar
simplices as the Poulsen simplex $\mathbb{P}$  plays in the category of
metrizable Choquet simplices (see next section \ref{Sec:systems}). Recall that a closed subset $H$ of a Lazar
simplex is a \emph{biface} or \emph{essential face} if it is the absolutely
convex hull of a (not necessarily closed) face \cite{lazar_banach_1971}.
This is equivalent to the assertion that the linear span of $H$ inside $%
A_{\sigma }(K)^{\ast }$ is a w*-closed $L$-ideal \cite%
{alfsen_structure_1972-1,alfsen_structure_1972-2}. Relevant properties of  $\mbb L$:
\begin{enumerate}[$\bullet$]
\item The Lusky simplex is the 
\emph{unique} nontrivial metrizable Lazar simplex with dense extreme
boundary (Lusky   \cite{lusky_gurarij_1976});
\item the Lusky simplex is  \emph{universal} among
metrizable Lazar simplices, in the sense that any metrizable Lazar simplex
is symmetrically affinely homeomorphic to a closed biface of $\mathbb{L}$ 
(Lusky \cite{lusky_construction_1979});
\item   the Lusky simplex is\emph{\
homogeneous}: any symmetric affine homeomorphism between proper closed
bifaces of $\mathbb{L}$ extends to a symmetric affine homeomorphisms of $%
\mathbb{L}$ (Lupini \cite[Subsection 6.1]{lupini_fraisse_2016}).
\end{enumerate}

Our intention is to prove the following:
\begin{theorem}\label{luskydmflkfskd}
Suppose that $H$ is a closed biface of the Lusky simplex $\mathbb{L}$. Then
the group $\mathrm{Aut}_{H}(\mathbb{L}) $ of symmetric affine homeomorphisms 
$\alpha$ of $\mathbb{L}$ such that $\alpha(p)=p$ for every $p\in H$ is
extremely amenable.
\end{theorem}

\begin{remark}
A similar result holds for complex Banach spaces. In this setting, one
considers compact convex sets endowed with a continuous action of the circle
group $\mathbb{T}$ (\emph{compact convex circled sets}). The compact convex
circled sets corresponding to complex Lindenstrauss spaces (\emph{Effros
simplices}) have been characterized by Effros in \cite{effros_class_1974}.
Again, the unit ball of the dual space of the complex Gurarij space has
canonical uniqueness, universality, and homogeneity properties within the
class of Effros simplices; see \cite[Subsection 6.2]{lupini_fraisse_2016}.
Here one considers the natural complex analog of the notion of a closed
biface (\emph{circled face}). The same argument as above shows that, in the complex case, the
pointwise stabilizer of any closed circled face of $\mathrm{Ball}(\mathbb{G}%
^{\ast })$ is extremely amenable.
\end{remark}

Observe that in the particular case when $H$ is the trivial biface $\left\{ 0\right\} $,
such a statement recovers extreme amenability of the group of surjective
linear isometries of $\mathbb{G}$.   Observe also that given a closed biface $H$ of a Lazar simplex $L$, we have that $g\in \mr{Aut}_H(L)$ if and only if $\widehat{g}\in \iso_{\widehat{i}}(A_\sig(L))$, where $i:H\rightarrow L$ is the inclusion map and where, in general, given  Banach spaces $X$ and $Y$ and an operator $\sigma :X\to Y$ by $\iso_\sigma (X)$ we mean the subgroup of isometries $g$ of $X$ so that $\sigma \circ g=\sigma $.   This motivates our study of such pairs $(X,\sigma )$. 
\begin{definition}[$R$-Banach space]
 Given a  Lindenstrauss space $R$, an {\em $R$-Banach space} is a  couple $\mbf X:=(X,\sig)$ when $\sig:X\to R$ is a linear contraction, called {\em $R$-functional}.   
\end{definition}  

In this category, given $R$-spaces $\mbf{X_0}:=(X_0,\sig_0)$, $\mbf{X_1}:=(X_1,\sig_1)$ and $\de>0$, let $\Emb_\de(\mbf{X_0},\mbf{X_1})$ be the collection of $\de$-isometric embeddings $\ga: X_0 \to X_1$ such that $\nrm{\sig_1 \circ \ga-\sig_0}\le \de$, and   in particular, let $\aut(\mbf X)=\iso_{\sig}(X)$ be the space of surjective isometries such that $\sigma \circ g=\sigma $. We write $(X_0,\sig_0)\con (X_1,\sig_1)$  to denote that $X_0\con X_1$ and that $\sig_1\rest X_0=\sig_0$.  The following result is
established in \cite[Section 5]{lupini_fraisse_2016}.
\begin{theorem}\label{jrio4jir4775jk}
Given a separable Lindenstrauss space $R$ there exists an onto contraction $\Omega_R:\mbb G\to R$ such that the $R$-Banach space $\pmb{\mbb G}_R:=(\mbb G,\Omega_R)$ is
\begin{enumerate}[1)]
\item {\em universal} for separable   $R$-Banach spaces, that is, for every such  space $\mbf X$, $\Emb(\mbf X, (\mbb G, \Omega_R))\neq \buit$; 
\item a {\em stable Fraïssé} $R$-Banach space with modulus of stability $\varpi(\de)=2\de$,   that is, for every finite-dimensional $R$-space $\mbf X:=(X,\sig)\con (\mbb G,\Omega_R)$,   every $\de>0$ and every $\ga\in \Emb_\de(\mbf X, (\mbb G,\Omega_R))$ there is an  isometry $g\in \iso_{\Omega_R}(\mbb G)$ such that  $\nrm{g\rest X- \ga}\le 2\de$.

\end{enumerate}

\end{theorem}  
Note that a  classical
result of Wojtaszczyk \cite{wojtaszczyk_remarks_1972} asserts that the
separable Lindenstrauss spaces are precisely the separable Banach spaces
that are isometric to the range of a contractive projection on the Gurarij
space $\mathbb{G}$.  
The $R$-functional $\Omega _{R}$   is called the \emph{generic}  
contractive $R$-functional on $\mbb G$. The name is justified by
the fact that the $\iso(\mbb G)$-orbit of $\Omega _{ R}$ is a dense $G_{\delta }$ subset of the space   of  contractive $R$-functionals on $\mbb G$.
The
universality and homogeneity properties of $\mathbb{L}$ can be seen as
consequences of the following result, established in \cite[Subsection 6.1]%
{lupini_fraisse_2016} using the theory of $M$-ideals in Banach spaces
developed by Alfsen and Effros \cite%
{alfsen_structure_1972-1,alfsen_structure_1972-2}, and the Choi--Effros
lifting theorem from \cite{choi_lifting_1977}.

\begin{proposition} \label{functionals_bifaces} 

Suppose that $R$ is a separable Lindenstrauss space. A contraction $s:\mbb G\to R$  belongs to the $\iso(\mbb G)$-orbit of $\Omega_R$  if and only if $s$ is a {\em non-trivial facial quotient}, that is,  if $\ker s\neq 0$, and $s^*$ is an isometric embedding such that $s^*(\ball(R^*))$ is a  closed biface of $\ball(\mbb G^*)$. 
%
%
%
%
%
\end{proposition}
In particular, suppose that $H$ is a proper closed biface of $\mbb L$, $i:H\to \mbb L$ is the canonical inclusion and we identify canonically $\mbb G$ and $A_\sig(\mbb L)$. Then $\widehat{i}: A_\sig(\mbb L)\to A_\sig(H)$ is a non-trivial facial quotient, hence $\widehat{i}\in \iso(\mbb G)\cdot \Om_{A_\sig(H)}$.  This implies that $\iso_{\widehat{i}}(\mbb G)= \iso_{\Om_{A_\sig(H)}}(\mbb G)$, and Theorem \ref{luskydmflkfskd} can be rephrased as follows.

\begin{theorem}\label{Corollary:ea-R-spaces}  The stabilizer of the generic contractive $R$-functional on the
Gurarij space is extremely amenable for any separable Lindenstrauss Banach space $R$.  \end{theorem}  
 
When $R=\{0\}$, we recover the extreme amenability of $\iso(\mbb G)$. In fact, the proof of this extension is based on the approximate Ramsey property of finite-dimensional $R$-Banach spaces, by means of  the KPT correspondence.  The corresponding non-commutative version of the previous theorem is established in \cite{bartosova_ramsey_2017}. 

\subsubsection{KPT correspondence and (ARP) of $R$-Banach spaces}\label{775939hhttti}

\label{Subs:extreme-functional} \label{Subs:ARP-functional}
We give a proof of Theorem \ref{luskydmflkfskd}. By the correspondence between the categories of Lazar simplices and that of $R$-Banach spaces, Theorem \ref{luskydmflkfskd} is equivalent to the fact that $\aut(\pmb{\mbb G}_R)$ is extremely amenable, which will be proved by means of a KPT correspondence and an appropriate approximate Ramsey property.  Given an $R$-space $\mbf X =(X,s)$, let $\age(\mbf X)$ be the collection of pairs $(F,s\rest F)$, where $F\in \age (X)$. Given a family $\mc F$ of finite-dimensional $R$-Banach spaces, let $[\mc F]$ be the collection of all separable $R$-Banach spaces $\mbf X$ such that for every $\mbf F\in\age(\mbf X)$ and every $\de>0$ there is some $\mbf G\in \mc F$ such that $\Emb_\de(\mbf F, \mbf G)\neq \buit$.     
 
\begin{theorem}[KPT correspondence for stable Fraïssé $R$-Banach spaces] \label{ARP=DARPforR-b}  Suppose that $\mbf E=(E,\Om)$ is an approximately ultrahomogeneous $R$-Banach space.  Then the following are equivalent 
\begin{enumerate}[1)]
\item $\aut(\mbf E)$ is extremely amenable.
\item $\mr{Age}(\mbf E)$  satisfies the  (ARP), that is    for every $\mbf X,\mbf Y\in \age(\mbf E)$ and $\vep>0$  there is $\mbf Z\in \age(\mbf E)$ such that every continuous coloring of $\Emb(\mbf X,\mbf Z)$ $\vep$-stabilizes on $\ga\circ \Emb(\mbf X,\mbf Y)$ for some $\ga\in \Emb(\mbf Y,\mbf Z)$.
\end{enumerate}
Suppose that  $\mc F$ is a family such that $\mc F\preceq \age(\mbb E)$, $\mbb E\in [\mc F]$. Then (1), (2), and (3) are   equivalent to

\begin{enumerate}[1)]\addtocounter{enumi}{2}
\item   $\mc F$ satisfies the {\em stably approximate Ramsey property (SRP)} with modulus $\varpi(\de)$, that is   for every $\mbf X,\mbf Y\in \mc F$, $\vep>0$ and $\de\ge 0$ there is $\mbf Z\in \mc F$ such that every  continuous coloring of $\Emb_\de(\mbf X,\mbf Z)$ $(\varpi(\de)+\vep)$-stabilizes on $\ga\circ \Emb_\de(\mbf X,\mbf Y)$ for some $\ga\in \Emb(\mbf Y,\mbf Z)$.

\end{enumerate}
\end{theorem}  
 The proof of Theorem \ref{ARP=DARPforR-b} is a straightforward modification of that of Theorem  \ref{liwjr3iwejirwe}; we leave its details to the reader.

\begin{theorem} \label{Theorem:ARP-R-spaces}  The following classes have the (SRP) with modulus of stability $2\de$:  
\begin{enumerate}[a)]
\item For every $k\in \N$, the class of $\ell_\infty^k$-Banach spaces $(X,s)$ where $X=\ell_\infty^n$ for some $n\in \N$.
\item For every  separable Lindenstrauss space  $R$  the  class of all finite-dimensional $R$-Banach spaces. 
\end{enumerate}

 \end{theorem}
\begin{proof}
 As for the case of Banach spaces in Proposition \ref{SRP=cSRP}, a class of $R$-finite dimensional spaces has the (SRP) with modulus $\varpi$ if and only if it satisfies the (ARP) and it has the corresponding stable amalgamation property with modulus $\varpi$. 
{\it a):}   
\clam \label{766jkhhuur} The family $\mc F$ of  $\ell_\infty^k$-spaces of the form $(\ell_\infty^n,s)$ for some $n$ has the stable amalgamation property with modulus $2\de$.
\fclam
\prucl
 Fix $\ell_\infty^k$-spaces $\mbf X=(\ell_\infty^d,s)$, $\mbf Y=(\ell_\infty^m,t)$ and $\mbf Z=(\ell_\infty^n,u)$, $\de>0$, and  $\ga\in \Emb_\de(\mbf X, \mbf Y)$ and $\eta \in \Emb_\de(\mbf X, \mbf Z)$.  Let $I\in \Emb(\ell_\infty^m, \ell_\infty^{m+n})$ and $J\in \Emb(\ell_\infty^n, \ell_\infty^{m+n})$ be such that $\nrm{I\circ \ga- J\circ \eta}\le \de$ (see Proposition \ref{li23jior3ji3r}). Then $I_0:=(I, t): \ell_\infty^m\to \ell_\infty^{m+n+k}$ and   $J_0:=(J, u): \ell_\infty^n\to \ell_\infty^{m+n+k}$ satisfies that $I_0\in \Emb(\mbf Y, (\ell_\infty^{m+n+k}, \pi))$,  $J_0\in \Emb(\mbf Z, (\ell_\infty^{m+n+k}, \pi))$  and $\nrm{I_0\circ \ga - J_0\circ \eta}=\max \{ \nrm{I\circ \ga - J \circ \eta}, \nrm{t\circ \ga- u\circ \eta}\}\le 2\de$, where $\pi:\ell_\infty^{m+n+k}\to \ell_\infty^k$ is the projection $\pi((a_j)_{j<m+n+k})=(a_j)_{j=m+n}^{m+n+k-1}$. 
\fprucl

 We prove now the (ARP) of $\mc F$.    Fix   $\ell_\infty^k$-spaces   $\mbf X:=(\ell_\infty^d, s)$  and $\mbf Y:=(\ell_\infty^m,u)$,  and  $\vep>0$. Let $n\in \N$ be witnessing  the (ARP) of $\{\ell_\infty^r\}_r$ for the initial parameters $d,m$,   and $\vep$. Let $\pi: \ell_\infty^{n+k}\to \ell_\infty^k$ be the canonical second projection $\pi((a_j)_{j<n+k}):= (a_j)_{j=n}^{n+k-1}$. We claim that $\mbf Z:=(\ell_\infty^{n+k},\pi)$ works:  For suppose that $c:\Emb(\mbf X,\mbf Z)\to [0,1]$ is a continuous coloring. Let $\widehat{c}: \Emb(\ell_\infty^d, \ell_\infty^n)\to [0,1]$ be defined for $\ga\in \Emb_\de(\ell_\infty^d,\ell_\infty^n)$ by $\widehat{c}(\ga)= c( \ga, s)$. Observe that $\widehat{c}$ is 1-Lipschitz, so there is $I\in \Emb(\ell_\infty^m,\ell_\infty^n)$ such that $\osc( \widehat{c}\rest I\circ \Emb(\ell_\infty^d,\ell_\infty^m))\le \vep$. Let $J:=(I, t)\in \Emb(\mbf Y, \mbf Z)$. Notice that given $\ga\in \Emb(\mbf X,\mbf Y)$,  we have that  $ J\circ \ga   =(I\circ \ga, t\circ \ga)= (I\circ \ga, s)$. Hence, $\osc( c\rest J\circ \Emb(\mbf X, \mbf Y ))\le \osc( \widehat{c}\rest I\circ \Emb(\ell_\infty^d,\ell_\infty^m))\le \vep$

{\it b)}:  Fix  a Lindenstrauss space $R$, and choose an increasing sequence of subspaces $(R_n)_n$ whose union is dense in $R$ and such that each $R_n$ is isometric to $\ell_\infty^n$. 
 
Let $\mc F$ be the class  of $R$-Banach spaces $(X,s)$ where $X$ is isometric to some $\ell_\infty^d$ and such that $\im s\con \bigcup_n R_n$. It follows easily from {\it a)} that $\mc F$  has the (SRP) with modulus $2\de$.  We know from Theorem \ref{jrio4jir4775jk} that $\pmb{\mbb G}_R=(\mbb G,\Omega_R)$  is a stable Fraïssé $R$-Banach space such that $\mr{age}( \pmb{\mbb G}_R)$ consists of all finite-dimensional $R$-Banach spaces. On the other hand, $ \pmb{\mbb G}_R\in [\mc F]$, so it follows from {\it a)} and the KPT correspondence in Theorem  \ref{ARP=DARPforR-b} that $\mr{age}(\pmb{\mbb G}_R)$ satisfies the (SRP) with modulus $2\de$. 
\end{proof}
Theorem \ref{Theorem:ARP-R-spaces} and the characterization of extreme
amenability in Theorem \ref{ARP=DARPforR-b}  give the previously announced result. 
\begin{reptheorem}{Corollary:ea-R-spaces}
The stabilizer of the generic contractive $R$-functional on the
Gurarij space is extremely amenable for any separable Lindenstrauss Banach space $R$. \qed
\end{reptheorem}

%

%
%

\section{The Ramsey property of Choquet simplices and function systems 
}\label{Sec:systems}

The main goal of this section is to establish the approximate (dual) Ramsey property for
Choquet simplices with a distinguished point. We will then apply this to compute the universal
minimal flow of the automorphisms group of the Poulsen simplex $\mbb P$. We will prove that the minimal compact $\mathrm{Aut}(%
\mathbb{P})$-space is the Poulsen simplex $\mathbb{P}$ itself endowed with
the canonical action of $\mathrm{Aut}(\mathbb{P})$, answering \cite[Question
4.4]{conley_fraisse_2017} (the fact that such an action is minimal is a
result of Glasner from \cite{glasner_distal_1987}).   This will be done by studying function systems with a
distinguished unital positive map to a fixed separable Lindenstrauss  function system $R$. 
%
%
%
%
Similarly as in the case of Banach spaces  (\S\S\ref{lio3j4oirjir4488}), we will also consider  function systems $X$  with a distinguished
{\em state}, a unital linear contraction $s: X \to R$ where $R$ is a  fixed  separable  Lindenstrauss function system.
%
%
\subsection{Choquet simplices and function systems\label%
{Subs:operator_systems}}

Recall that a \emph{compact convex set }$K$\emph{\ }is a compact convex
subset of some locally convex topological vector space. In a compact convex
set one can define in the usual way the notion of convex combination. The 
\emph{extreme boundary }$\partial _{e}K$ of $K$ is the set of \emph{extreme
points} of $K$, that is, points that cannot be written in a nontrivial way
as a convex combination of points of $K$. When $K$ is metrizable the
boundary $\partial _{e}K$ is a $G_{\delta }$ subset. In this case, a \emph{%
boundary measure }on $K$ is a Borel probability measure on $K$ that vanishes
off the boundary of $K$. Choquet's representation theorem asserts that any
point in a compact convex set can be realized as the barycenter of a
boundary measure on $K$ (\emph{representing measure}). A compact convex set $%
K$ where every point has a \emph{unique }representing measure is called a 
\emph{Choquet simplex}. In particular, any standard finite-dimensional
simplex $\Delta _{n}$ for $n\in \mathbb{N}$ is a Choquet simplex.

The class of standard finite-dimensional simplices $\Delta _{n}$ for $n\in 
\mathbb{N}$ naturally form a {\em projective} Fra\"{\i}ss\'{e} class in the sense
of \cite{irwin_projective_2006}; see \cite{kubis_lelek_2015}. The
corresponding Fra\"{\i}ss\'{e} limit is the \emph{Poulsen simplex} $\mathbb{P%
}$. Initially constructed by Poulsen in \cite{poulsen_simplex_1961}, $%
\mathbb{P}$ is a nontrivial metrizable Choquet simplex with dense extreme
boundary.\ It was later shown in \cite{lindenstrauss_poulsen_1978} that
there exists a \emph{unique }nontrivial metrizable Choquet simplex with this
property up to affine homeomorphism. Furthermore $\mathbb{P}$ is \emph{%
universal }among metrizable Choquet simplices, in the sense that any
metrizable Choquet simplex is affinely homeomorphic to a closed proper face
of $\mathbb{P}$. Also, the Poulsen simplex is ultrahomogeneous: any affine
homeomorphism between closed proper faces of $\mathbb{P}$ extends to an
affine homeomorphism of $\mathbb{P}$.

The Poulsen simplex $\mathbb{P}$ can also be studied from the perspective of
direct Fra\"{\i}ss\'{e} theory by considering the natural dual category to
compact convex sets. For a compact convex set $K$, let $A(K)$ be the space of
complex-valued continuous affine functions on $K$. This is a closed subspace
of the space $C(K)$ of complex-valued continuous functions on $K$, endowed
with the supremum norm. Furthermore, $A(K)$ contains a distinguished element,
its \emph{unit}, that corresponds to the constant function equal to  $1$.
In general, recall  that a \emph{function system} is a closed subspace $V$ of $C(T)$ for
some compact Hausdorff space $T$ containing the function constantly equal to 
$1$ and such that if $f\in V$ then the function $f^{\ast }$ defined by $%
f^{\ast }\left( t\right) =\overline{f\left( t\right) }$ also belongs to $V$.
So,\ $A(K)$ is a function system, and in fact any function system $%
V\subseteq C(T)$ arises in this way from a suitable compact convex set $K$.
Precisely, $K$ is the compact convex set of \emph{states }of $V$, that is,
the contractive functionals on $V$ that are \emph{unital}, i.e., that map the
unit of $C(T)$ to $1$.

As mentioned in the introduction, the assignment $K\mapsto A(K)$ establishes
a contravariant equivalence of categories from the category of compact
convex sets and continuous affine maps to the category of function systems
and unital contractive linear maps. The finite-dimensional function systems
that are \emph{injective} in such   category are precisely the function
systems $A(\Delta _{n})=\ell _{\infty }^{n}$ corresponding to the standard
finite-dimensional simplices $\Delta_{n}$. The function systems that
correspond to Choquet simplices are precisely those that are Lindenstrauss
as Banach spaces, or equivalently, the function systems whose identity map
is the pointwise limit of \emph{unital} contractive linear maps that factor
through finite-dimensional injective function systems.

The function systems approach has been adopted in the work of Conley and T%
\"{o}rnquist \cite{conley_fraisse_2017} and, independently, in \cite%
{lupini_fraisse_2016,lupini_fraisse_2017}, where it is shown that the class
of finite-dimensional function systems is a Fra\"{\i}ss\'{e} class. Its
limit can be identified with the function system $A(\mathbb{P})$
corresponding to the Poulsen simplex, which we will call the \emph{Poulsen
system}.
 The model-theoretic properties of $A(%
\mathbb{P})$ and their non-commutative analogues   have been
studied in \cite{goldbring_model-theoretic_2015}.


 Suppose that $X$ is an function system. Recall that a \emph{%
state} on $X$ is a unital contractive linear map from $X$ to $\mathbb{C}$. More generally,   if $R$
is any separable  Lindenstrauss function system, we call a unital contractive linear map from $X$ to $R$
an {\em $R$-state} on $X$.  Let $\mathrm{UC}(X,R)$  be the space of $R$-states on $X$. We have that $\mathrm{UC}(X,R)$ is a Polish space endowed with a canonical continuous action
of $\mathrm{Aut}(X)$. An \emph{$R$-function system} is a pair $\mathbf{X}%
=(X,s_{X})$ of a function system $X$ and an $R$-state $s_{X}$ on $X$. In the
following, we regard $\mathrm{UC}\left( X,R\right) $ as an $\mathrm{Aut}%
\left( X\right) $-space with respect to the canonical action $\mathrm{Aut}%
\left( X\right) \curvearrowright \mathrm{UC}\left( X,R\right) $ given by $%
\left( \alpha ,s\right) \mapsto s\circ \alpha ^{-1}$. 


Given $R$-function systems $\mathbf X=(X,s_X)$ and $\mathbf Y=(Y,s_Y)$ and given $\de\ge0$, let $\Emb_\de(\mathbf X, \mathbf  Y)$ be the collection of unital $\de$-isometric embeddings $\ga:X\to Y$ such that $\nrm{s_Y\circ \ga -s_X}_{X,R}\le \de$.   Given an $R$-function system $\mbf X=(X,s_X)$, let $\age(\mbf X)$ be the collection of all finite-dimensional {\em $R$-function subsystems} $\mbf Y=(Y,s_Y)$  of $\mbf X$, that is, $Y\con X$ and $s_Y=s_X\rest Y$.  Given a class $\mc F$ of $R$-function systems,   let $[\mc F]$   be the class of all separable $R$-function systems $\mbf X=(X,s_X)$ such that for every $\mbf Y$ and every $\de>0$ there is $\mbf Z\in \mc F$ such that $\Emb_\de(\mbf Y, \mbf Z)\neq \buit$. 
Let $\mathrm{Aut}%
(X,s_{X})$ be the stabilizer of $s_{X}\in \mathrm{UC }\left( X,R\right) $ in 
$\Aut(X)$. Given a family $\mathcal{A}$ of function systems, let $\mathcal{A}%
^{R}$ be the collection of $R$-function systems $(X,s_{X})$ where $X\in 
\mathcal{A}$.

The following result is established in \cite[\S\S6.3]{lupini_fraisse_2016}.

\begin{proposition}
\label{Proposition:fraisse-state}Let  $R$ be a separable Lindenstrauss function system.  Then
the class $\mr{FdBa}^R $ of finite-dimensional $R$-function systems is a  stable Fra\"{\i}ss\'{e} class with stability modulus $\varpi (\delta )=2\delta $ and  $\posy_R:=(A(\mbb P),\Omega _{R})$ is its {\em Fraïssé limit}, that is, $\posy_R$ is a  stable Fra\"{\i}ss\'{e}   $R$-function
system such that $\age(\posy_R)=\mr{FdBa}^R$. 
\end{proposition}

As in the case of operator spaces, the $R$-state $\Omega _{ R}$ as in Proposition \ref{Proposition:fraisse-state} is called the \emph{%
generic} $R$-state on $A(\mbb P)$. This is the unique $R$-state on $A(\mbb P)$ whose $\mathrm{Aut}(A(\mbb P))$-orbit
is a dense $G_{\delta }$ subset of the space $\mathrm{UC}(A(\mbb P)
,R) $. The elements of the $\mathrm{Aut}(A(\mbb P))$-orbit of $\Omega
_{R}$ can be characterized as follows (see \cite[\S\S6.3]{lupini_fraisse_2016}.

\begin{proposition} \label{fffunctionals_bifaces} 

Suppose that $R$ is a separable Lindenstrauss function system.   A unital quotient map  $s:A(\mbb P)\to R$  belongs to the  $\aut(A(\mbb P)) $-orbit of $ \Om_R$ if and only if $s$ is a {\em unital facial quotient}, i.e., $s$ is unital and $s^*$ is an isometric embedding such that $s^*(\ball(R^*))$ is a  closed proper face of $\mbb P$. 
\end{proposition}

The intention is to prove the following 
\begin{theorem}
\label{Corollary:ea-osy}  For every    metrizable Choquet simplex $F$  the stabilizer $\mathrm{Aut}(\posy_{A(F)})$ of the generic $A(F)$-state $\Omega _{A(F)}$ on
the Poulsen system $A(\mathbb{P})$ is extremely amenable.
\end{theorem}

\subsection{Approximate Ramsey property and extreme amenability}\label{ui77wwwoooiii}

The following result provides a correspondence between extreme amenability and Ramsey properties in the context of $R$-function systems. The proof is analogous to the one for Banach spaces, and is left to the reader.

\begin{theorem}[KPT correspondence for (aUH)  and stable Fraïssé $R$-function systems]\label{liwjr3iwejirweaaa} Suppose that $\mbf X=(X,\Om)$ is an approximately ultrahomogeneous $R$-function system. Then the following are equivalent:
\begin{enumerate}[1)]
\item $\aut(\mbf X)$ is extremely amenable.
\item $\age(\mbf X)$  satisfies the approximate Ramsey property.
\end{enumerate}
If in addition $\mc F$ is  a family that satisfies the stable amalgamation property such that $E\in [\mc F]$ and   $\mc F\preceq \age(\mbf X)$, that is, every $R$-function system in $\mc F$ can be isometrically $R$-embedded  into $E$,  then the previous are  equivalent to 
\begin{enumerate}[1)]\addtocounter{enumi}{2}
\item   $\mc F$ satisfies the (SRP).
  \end{enumerate}

\end{theorem}

\begin{theorem}\label{Proposition:Ramsey-systems}
Suppose that $(R_k)_k$ is a sequence of function subsystems of $R$, each $R_k$ isometric to $\ell_\infty^k$, and with a dense union.  The following classes of $R$-function systems have the (SRP) with modulus $2\de$:
\begin{enumerate}[1)]
\item For every $k\in \N$ the class of $R_k$-function systems $(R_n\oplus_\infty R_k, \pi_n^{(k)})$, where $\pi_n^{(k)}: R_n\oplus_\infty R_k\to R_k$ is the canonical second projection.
\item  For every $k\in \N$ the class of $R_k$-function systems $(X,s)$ where $X$ is isometric to some $\ell_\infty^n$. 
\item The class $\mc B_R$ of $R$-function systems $(X,s)$ where $X$ is isometric to some $\ell_\infty^n$ and $s: X\to \bigcup_k R_k$.  
\item The class of all $R$-function systems. 

\end{enumerate}   

\end{theorem}  

To prove this Theorem we will use (and prove) the (ARP) of the class $\{\ell_\infty^n\}_n$ with respect to positive embeddings.  Its proof is similar to that of  Lemma  \ref{iojriowjrwe8855}. We present the details for the reader's
convenience. Let $\Emb^+(\ell_\infty^d,\ell_\infty^n)$ be the space of {\em positive} isometric embeddings from $\ell_\infty^d$ into $\ell_\infty^n$.  Dually, let $\mr{Quo}^+(\ell_1^n,\ell_1^d)$ be the space of corresponding positive quotient mappings. 
\begin{lemma}\label{bbnvjkyurie}
 For every $d,m,r\in \N$,  and $\vep>0$ there is some $n$ such that every  $r$-coloring of $\Emb^+(\ell_\infty^d,\ell_\infty^n)$ has an $\vep$-monochromatic set of the form $\ga\circ \Emb^+(\ell_\infty^d,\ell_\infty^n)$ for some $\ga\circ \Emb^+(\ell_\infty^d,\ell_\infty^m)$.    
\end{lemma}

 We write $\ball^+(\ell_1^k)$  and $\mr{Sph}^+(\ell_1^k)$  to denote the positive part of the unit ball  and of the unit sphere of  $\ell_1^k$, respectively.  
Recall that a  linear map $\ga :\ell _{\infty }^{d}\rightarrow \ell _{\infty
}^{n}$ is unital if  and only if its dual $\ga^{\ast }:\ell
_{1}^{n}\rightarrow \ell _{1}^{d}$ is trace-preserving, that is, $\Tr_d(\ga^*((a_j)_{j<n}))= \Tr_n((a_j)_{j<n})$, where $\Tr_k ((a_j)_{j<k}):= \sum_{j<k} a_j$ is the canonical    trace. When in addition $\ga$ (equiv. $\ga^*$) is a contraction, then $\ga$ and $\ga^*$ must be positive.
Thus $\ga$ is a unital isometric embedding    if and
only if $\ga ^{\ast }$ is a trace-preserving   
 quotient mapping, or equivalently  if each $\ga^*(u_j)$ belongs to    $\mr{Sph}^+(\ell_1^d)$   and $\{u_j\}_{j<d}\con \{\ga^*(u_j)\}_{j<n}$.    
Given $R$-function systems $(\ell_\infty^d,s)$ and $(\ell_\infty^n,t)$, let $\Quo((\ell_1^n,t^*),(\ell_1^d,s^*))$ be the space of trace-preserving quotients $\sig:\ell_1^n\to \ell_1^d$ such that $\sig\circ t^*=s^*$.  Before proving Lemma \ref{bbnvjkyurie}, we use it. 

\begin{proof}[Proof of Theorem \ref{Proposition:Ramsey-systems}.]
All the four classes considered have the stable amalgamation property with modulus $2\de$: For the first three ones,  the  proof of Claim \ref{766jkhhuur} can be easily adjusted to give the desired property, and for the last class, as we mentioned above, the proof can be found in \cite[\S\S6.3]{lupini_fraisse_2016}. So, we have to prove that in addition all that classes have the (ARP). 

{\it 1)}:  We consider the equivalent class, and easy to work with, $\{(\ell_\infty^{n+k},\pi_n)\}_n$, where $\pi_n^{(k)}: \ell_\infty^{n+k}=\ell_\infty^n\oplus_\infty \ell_\infty^k  \to \ell_\infty^k$ is the second projection.   We prove the dual approximate Ramsey statement for the corresponding dual class: Write $\mbf X_n$ to denote $(\ell_\infty^{n+k}, \pi_n^{(k)})$, and $\mbf X_n^*:=(\ell_1^{n+k}, (\pi_n^{(k)})^*)$. Notice that given $d$ and $m$,  we have that    $\sig\in\Quo(\mbf X_m^*,\mbf X_d^*)$ exactly when $\sig$ is a trace-preserving quotient   such that $\sig(u_{m+j})=u_{d+j}$ for all $j<k$.  We prove that for all $d,m,r\in \N$ and $\vep>0$ there is some $n$ such that every $r$-coloring of $\Quo(\mbf X_n^*,\mbf X_d^*)$ has a $\vep$-monochromatic set of the form $\Quo(\mbf X_m^*,\mbf X_d^*)\circ \ro$ for some $\ro\in \Quo(\mbf X_n^*,\mbf X_m^*)$. Fix parameters $d,m,r$ and $\vep$. We claim that   the number $n$ obtained by applying Lemma  \ref{bbnvjkyurie} to $d+k-1$, $m+k-1$, $r$,$\de=0$, and $\vep$ works.   For suppose that $c:\Quo(\mbf X_n^*,\mbf X_d^*)\to r$. We define the auxiliary coloring $\widehat{c}:\Quo^+(\ell_1^{n},\ell_1^{d+k-1})\to r$ by declaring $\widehat{c}(\sig):=c(\widehat{\sig})$, where $\widehat{\sig}\in \Quo(\mbf X_n^*,\mbf X_d^*)$ is  such that $\widehat{\sig}(u_j)= i_{d}(\sig(u_j)) + (1-\nrm{\sig(u_j)}_1) u_{d+k-1}$, for $i_{d}: \ell_1^{k+d-1}\to \ell_1^{(k+d}$ being the canonical embedding $i_d((a_j)_{j<k+d-1}):=(a_0,\dots,a_{k+d-2},0)$.   Notice that   
 $\nrm{\sig-\eta}\le \nrm{\widehat{\sig}-\widehat{\eta}}\le 2\nrm{\sig-\eta}$.
%
%
By the choice of $n$, and the dual version of Lemma  \ref{bbnvjkyurie},  we can find $\ro\in \Quo^+(\ell_1^n,\ell_1^{m+k-1})$ and $\widehat{r}<r$ be such that $\Quo^{+}(\ell_1^{m+k-1},\ell_1^{d+k-1})\circ \ro \con (\widehat{c}^{-1}(\widehat{r}))_{\vep}$.  Let $\widehat{\ro}\in \Quo( \mbf X_n^*, \mbf X_m^*)$ be defined linearly for $j<n$ by $\widehat{\ro}(u_j):= i_{m} (\ro( u_j))+ (1- \nrm{\ro(u_j)}_1) u_{m-1}$.  We claim that $\Quo( \mbf X_m^*,\mbf X_d^*)  \circ \widehat{\ro}\con c^{-1}(\widehat{r})_{ 2\vep}$, so to this end, we fix $\sig\in \Quo( \mbf X_m^*,\mbf X_d^*)$. Let $\underline{\sig}: \ell_1^{m+k-1}\to \ell_1^{d+k-1}$, $\underline{\sig}(a_0,\dots,a_{m+k-2})=\pi(\sig(a_0,\dots,a_{m+k-2},0))$, where $\pi(b_0,\dots, b_{d+k-2},b_{d+k-1})= (b_0,\dots, b_{d+k-2})$.   It follows that $\underline{\sig}\in \Quo^+(\ell_1^{m+k-1},\ell_1^{d+k-1})$ and that $\widehat{\underline{\sig}\circ \ro}= \sig \circ \widehat{\ro}$. Let $\eta \in \Quo^+(\ell_1^n,\ell_1^{d+k-1})$ be such that $\nrm{\underline{\sig}\circ \ro- \eta}\le \vep$ and $c(\widehat{\eta})=\widehat{r}$. Then,  $\nrm{\widehat{\underline{\sig}\circ \ro}-\widehat{\eta}}\le 2\vep$, or, equivalently, $\nrm{\sig \circ \widehat{\ro}- \widehat{\eta}}\le 2\vep$.     
 
 {\it 2)}:  We prove the (ARP) of the equivalent class of $\ell_\infty^k$-function systems  $(\ell_\infty^n, s)$ for some $n\in \N$. Fix $\mbf X=(\ell_\infty^d,s)$, $\mbf Y=(\ell_\infty^m,t)$, $r\in \N$ and $\vep>0$, and we use the (ARP) in {\it 1)} to these parameters to find the corresponding $n$. We claim that $\mbf X_n=(\ell_\infty^{n+k},\pi_n^{(k)})$ works.  Let $c:\Emb(\mbf X,\mbf X_n)\to r$, and let $\widehat{c}:\Emb(\mbf X_d,\mbf X_n)\to r$ be defined by $\widehat{c}(\ga):= c(\ga\circ i)$, where $i\in \Emb(\mbf X, \mbf X_d)$ is defined by $i(x):=(x,s(x))$.
 Let $\ga\in \Emb(\mbf X_m,\mbf X_n) $ and $\widehat{r}<r$ be such that $\ga\circ\Emb( \mbf X_d,\mbf X_m) \con (\widehat{c}^{-1}(\widehat{r}))_\vep$. Let $\ga_0:=\ga\circ j\in \Emb(\mbf Y, \mbf X_n)$ where $j\in \Emb(\mbf Y, \mbf X_m)$ is $j(y):=(y,t(y))$. Then, $\ga_0\circ \Emb(\mbf X,\mbf Y)\con (c^{-1}(\widehat{r}))_\vep$, because  given $\eta\in \Emb(\mbf X,\mbf Y)$, if we define $\eta_0\in \Emb(\mbf X_d,\mbf X_m)$ by $\eta_0(x,y)=(\eta(x),  y)$,  then  we have that   $\eta_0 \circ i= j\circ \eta$, hence,   $\ga_0 \circ \eta= \ga\circ j \circ \eta= (\ga\circ \eta_0)\circ i$, and   consequently $c(\ga_0\circ \eta)=\widehat{c}(\ga \circ \eta_0)$. 
 
 {\it 3)} 
 follows trivially from {\it 2)}. {\it 4)} follows from {\it 3)} and the next. 
 \clam 
For every finite-dimensional $R$-function system $\mbf X$ and every $\de$ there is some $\mbf Y\in \mc B_R$ such that $\Emb_\de(\mbf X,\mbf Y)\neq \buit$. 
\fclam
\prucl
Since a function space is a unital closed subspace of some function system, it follows, for example from the existence of partitions of unity, that for every finite-dimensional function system $X$ and every $\de$ there is some $n$ and some unital $\de$-embedding $\ga: X\to \ell_\infty^n$. If in addition $s: X\to R$ is a unital contraction, then there must be some $k$ and some unital contraction $t: X\to R_k$ such that $\nrm{s-t}\le \de$. Being $\ell_\infty^k$ an injective function system,  we can find a unital contraction $u:\ell_\infty^n\to \ell_\infty^k$ such that $\nrm{u\circ  \ga-t}\le \de$, and consequently, $\ga\in \Emb_{2\de}((X,s), (\ell_\infty^n, u))$.  
\fprucl

\end{proof}

\begin{proof}[Proof of Lemma \ref{bbnvjkyurie}.] 
The proof of the dual form of this statement is that of Lemma \ref{iojriowjrwe8855}  with the natural modifications that we pass to sketch: Fix $d,m,r$ and $\vep$. 
First of all, fix  a finite $\vep$-dense subset $\mc D$ of $\ball^+(\ell_1^d)$ containing $\{u_j\}_{j<d}$ such that for every non-zero $x\in \ball^+(\ell_1^d)$ there is $y\in \mc D$ such that $\nrm{y-x}_1\le \vep$ and $\nrm{y}_1<\nrm{x}_1$. Let $\emb(d,m)$ be the collection of $1-1$ mappings from $d$ into $m$, and for each such mapping $f$, let $h_f: \ell_1^d\to \ell_1^m$ be the positive isometry sending $u_j$ to $u_{f(j)}$. Observe that for each  positive quotient mapping $\sig \in \Quo^+(\ell_1^m,\ell_1^d)$ there is some $f\in \emb(d,m)$ such that $\sig \circ h_f =\id_{\ell_1^d}$.  Let $\De:=\mc D \times \emb(d,m)$.  We linearly order $\mc D$  by $\prec$ in a way that if $\nrm{x}_1<\nrm{y}_1$, then $x\prec y$,  $\emb(d,m)$ arbitrarily, and then we consider $\De$ lexicographically ordered. Then $n:=  \mathbf{GR}(|\mc D|, |\De|,r)$ works.  Given $c:\Quo^+(\ell_1^n, \ell_1^d)\to r$ one defines $\widehat{c}: \mr{Epi}(n, \mc D)\to r$, $\widehat{c}(\sig):= c(\Phi(\sig))$ where $\Phi(\sig)(u_j):= \sig(j)$ for $j<n$. Let $\ro\in \mr{Epi}(n,\De)$ and $\widehat{r}<r$ be such that $ \mr{Epi}(\De,\mc D)\circ \ro$ is monochromatic with color $\widehat{r}$. Let $e\in \Quo^+(\ell_1^n,\ell_1^m)$ be linearly defined by $e(u_j):= h_f(v)$, where $(v,f)=\ro(j)$.  Then it can be shown as in the proof of Lemma \ref{iojriowjrwe8855}  that
for every $ T\in \Quo^+(\ell_1^m,\ell_1^d)$ there is some $\sig\in \mr{Epi}(\De, \mc D)$ such that $\nrm{\Phi(\sig\circ \ro)- T\circ e}\le \vep$, and consequently,  $\Quo^+(\ell_1^m,\ell_1^d)\circ e\con (c^{-1}(j))_{\vep}$. 
\end{proof}
From Theorem \ref{Proposition:Ramsey-systems} and the (KPT) correspondence in Theorem \ref{liwjr3iwejirweaaa}  we obtain the following.

\begin{reptheorem}{Corollary:ea-osy} 
 For every   any metrizable Choquet simplex $F$  the stabilizer $\mathrm{Aut}(\posy_{A(F)})$ of the generic $A(F)$-state $\Omega _{A(F)}$ on
the Poulsen system $A(\mathbb{P})$ is extremely amenable. \qed

\end{reptheorem}

We rephrase  {\it 1)} of Theorem   \ref{Proposition:Ramsey-systems}   geometrically. We
identify the $n$-dimensional standard simplex $\Delta _{n}$ with the state
space $S(\ell _{\infty }^{n+1})\subset \ell _{1}^{n+1}$. Let $\mathrm{Epi}%
(\Delta _{n},\Delta _{d})$ be the space of surjective continuous affine maps
from $\Delta _{n}$ to $\Delta _{d}$ endowed with the metric $d(\phi ,\psi
)=\sup_{p\in \Delta _{n}}\left\Vert \phi (p)-\psi (p)\right\Vert _{\ell
_{1}^{d}}$. We also let $\mathrm{Epi}_{0}(\Delta _{n},\Delta _{d})$ be the
subspace of $\phi \in \mathrm{Epi}(\Delta _{n},\Delta _{d})$ such that $\phi
(u_n)=u_d$. One can (isometrically) identify $\mathrm{Epi}(\Delta
_{n},\Delta _{d})$ isometrically with the space of trace-preserving quotients from $\ell_1^n$ onto $\ell_1^d$, and the space $\mathrm{Epi}_{0}(\Delta _{n},\Delta _{d})$ with $\Quo((\ell_1^n, \pi_n^{(1)}),(\ell_1^d, \pi_n^{(1)}))$. The following statement is
therefore equivalent to Theorem \ref{Proposition:Ramsey-systems} for $k=1$ and isometric embeddings.

\begin{corollary}
For any $d,m,r\in \mathbb{N}$ and $\varepsilon >0$ there exists $n\in 
\mathbb{N}$ such that for any $r$-coloring of the space $\mathrm{Epi}%
_{0}(\Delta _{n},\Delta _{d})$ there exists $\gamma \in \mathrm{Epi}%
_{0}(\Delta _{n},\Delta _{m})$ such that $\mathrm{Epi}_{0}(\Delta
_{m},\Delta _{d})\circ \gamma $ is $\varepsilon $-monochromatic. \qed
\end{corollary}

\subsection{Closed faces of the Poulsen simplex\label{Subsection:faces}}

Theorem  \ref{Corollary:ea-osy} can be restated geometrically in terms of a
property of the Poulsen simplex. The Poulsen simplex $\mathbb{P}$ has the
following universality and homogeneity property for faces: any metrizable
Choquet simplex is affinely homeomorphic to a closed proper face of $\mathbb{%
P}$, and any affine homeomorphism between closed proper faces of $\mathbb{P}$
extends to an affine homeomorphism of $\mathbb{P}$ \cite[Theorem 2.3 and
Theorem 2.5]{lindenstrauss_poulsen_1978}. This can be seen as a consequence
of the following geometric version of Proposition \ref{fffunctionals_bifaces}.  

\begin{proposition}
\label{states_faces} Let $F$ be a metrizable Choquet simplex, and let $s:A(\mbb P)\to A(F)$ be a unital quotient. The following
assertions are equivalent:

\begin{enumerate}[1)]
\item $s$ belongs to the $\mathrm{Aut}(\mathbb{P})$-orbit of the generic $%
A(F)$-state $\Omega_{A(F)}$.

\item There is a closed proper face $\bar{F}$ of $\mathbb{P}$ affinely
homeomorphic to $F$   such that $s$ is the restriction map $A(\mathbb{P})\rightarrow A(%
\bar{F})$, $f\mapsto f\upharpoonright _{\bar{F}}$.
\end{enumerate}
In particular, the $\mathrm{Aut}(\mathbb{P})$-orbit of the generic state $%
\Omega _{\mathbb{R}}:A(\mathbb{P})\rightarrow \mathbb{R}$ is
the set of extreme points of $\mathbb{P}$.
\end{proposition}
Hence, Theorem  \ref{Corollary:ea-osy}   can be reformulated as follows.

\begin{theorem}
\label{Theorem:faces}Suppose that $F$ is a closed proper face of the Poulsen
simplex $\mathbb{P}$. Then the pointwise stabilizer of $F$ with respect to
the canonical action $\mathrm{Aut}(\mathbb{P})\curvearrowright \mathbb{P}$
is extremely amenable.\qed 
\end{theorem}

%
\subsection{The universal minimal flows of $\mathbb{P}$  }\label{oi5t8t895466655}

Using Theorem \ref{Theorem:faces} we
can compute the universal minimal flow of the affine homeomorphism group $%
\mathrm{Aut}(\mathbb{P})$ of the Poulsen simplex. 
\begin{theorem}
\label{Theorem:flows} The universal minimal flow of $\mathrm{Aut}( \mathbb{P}) $ is the
canonical action $\mathrm{Aut}( \mathbb{P}) \curvearrowright \mathbb{P}$.
 
\end{theorem}

\begin{proof}
The action $\mathrm{Aut}(\mathbb{P})\curvearrowright 
\mathbb{P}$ is minimal by a result of Glasner from \cite{glasner_distal_1987}%
. This can be seen directly using the homogeneity property of $A(\mathbb{P})$
and the fact that for any $\varepsilon >0$ and $d\in \mathbb{N}$ there
exists $m\in \mathbb{N}$ such that for any state $s$ on $\ell _{\infty }^{d}$
and $t$ on $\ell _{\infty }^{m}$ there exists a unital linear isometry $\phi
:\ell _{\infty }^{d}\rightarrow \ell _{\infty }^{m}$ such that $\left\Vert
t\circ \phi -s\right\Vert <\varepsilon $. Consider the generic state $\Omega
_ {\mathbb{R}}$ on $A(\mathbb{P})$. We know from Proposition %
\ref{states_faces} that the state $\Omega _ {\mathbb{R}}$ is
an extreme point of $\mathbb{P}$, whose $\mathrm{Aut}(\mathbb{P})$-orbit is
dense. The stabilizer $\mathrm{Aut}(\posy_ {%
\mathbb{R}}) $ of $\Omega _ {\mathbb{R}}$ is extremely
amenable by Theorem \ref{Corollary:ea-osy}. The canonical $\mathrm{Aut}(%
\mathbb{P})$-equivariant map from the quotient $\mathrm{Aut}(\mathbb{P})$%
-space $\mathrm{Aut}(\mathbb{P})\quo\mathrm{Aut}(\posy_ {\mathbb{R}})$ to $\mathbb{P}$ is a uniform
equivalence. This follows from the homogeneity property of $\posy _ {\mathbb{R}}$ as the Fra\"{\i}ss\'{e} limit of
the class of finite-dimensional function systems with a distinguished state;
see also \cite[Subsection 5.4]{lupini_fraisse_2016}. This allows one to
conclude via a standard argument---see \cite[Theorem 1.2]%
{melleray_polish_2016}---that the action $\mathrm{Aut}(\mathbb{P}%
)\curvearrowright \mathbb{P}$ is the universal minimal compact $\mathrm{Aut}(%
\mathbb{P})$-space.
\end{proof}
The universal minimal flows for the non-commutative versions of the Poulsen simplex have been computed in \cite{bartosova_ramsey_2017}.
It has recently been proved in \cite{ben_yaacov_metrizable_2017} that the
situation in Theorem \ref{Theorem:flows} is typical. Whenever $G$ is a
Polish group whose universal compact $G$-space $M( G) $ is metrizable, there
exists a closed extremely amenable subgroup $H$ of $G$ such that the
completion of the homogeneous quotient $G$-space $G/H$ is $G$-equivariantly
homeomorphic to $M(G) $.


\bibliographystyle{amsplain}
\bibliography{bibliography}

\end{document}